\documentclass[11pt,dvips]{amsart}  
\usepackage{amssymb,amsmath,amsfonts}  
\usepackage[dvips]{graphicx,xcolor}
\usepackage{cite}
\allowdisplaybreaks[4]
\definecolor{pgray}{gray}{0.4}

\newtheorem{theorem}{Theorem}[section]
\newtheorem{corollary}[theorem]{Corollary}
\newtheorem{proposition}[theorem]{Proposition}
\newtheorem{lemma}[theorem]{Lemma}

\newtheorem{remark}[theorem]{Remark}
\numberwithin{equation}{section}

\title{\mbox{}}

\begin{document}
\begin{center}
{\bf \LARGE{
	Nonlinear profiles on solutions to the surface quasi-geostrophic equation
}}\\
\vspace{5mm}
{\large Masakazu Yamamoto}\\Graduate School of Science and Technology, Gunma University\\Tenjin-cho 1-5-1, Kiryu 376-8515, Japan \\E-mail : \texttt{mk-yamamoto@gunma-u.ac.jp}
\end{center}
\maketitle
\vspace{-15mm}
%
\begin{abstract}
The quasi-geostrophic equation is well known as a model for predicting potential temperature around low-pressure systems in high-latitude regions.
When diffusion effects are added to this equation, it serves as a model for potential temperature at the sea surface.
In either case, the nonlinear term represents the effect of the Coriolis force.
This paper yields nonlinear profile in solutions to the surface quasi-geostrophic equation.
This profile is determined uniquely from the perspective of large-time behavior of solutions.
Components of solutions are arranged sequentially from slow to fast decay based on the parabolic scale.
Then such an expansion is determined uniquely.
Since the equation is subcritical in the context of large-time effects, the main components exhibit linear features.
Nonlinear characteristics appear in components that have a smaller effect on solutions.
From the perspective of the correlation between spacetime variables, nonlinear profiles exhibit characteristics that are distinctly different from those of linear profiles.
Moreover, the nonlinear effects of this equation are expected to be weaker than those of other equations with the same scale due to their spatial structure.
The main theorem publishes minute, but unique rotational flow arising from the Coriolis force.

\vspace{2mm}

Coriolis force,~
Atmospheric models,~
Oceanic models
\end{abstract}
\section{Introduction}
We treat the Cauchy problem of the following surface quasi-geostrophic equation:
\begin{equation}\label{qg}
\left\{
\begin{array}{lr}
	\partial_t \theta + \mathcal{R}^\bot\theta\cdot\nabla \theta = \Delta \theta,
	&
	t > 0,~ x \in \mathbb{R}^2,\\
	\theta (0,x) = \theta_0 (x),
	&
	x \in \mathbb{R}^2,
\end{array}
\right.
\end{equation}
where $\partial_t = \partial/\partial t$ and $\mathcal{R}^\bot = (-\mathcal{R}_2, \mathcal{R}_1)$ for $\mathcal{R}_j \varphi = \partial_j (-\Delta)^{-1/2} \varphi = \mathcal{F}^{-1} [i\xi_j\mathcal{F} [\varphi]/|\xi|]$ is the orthogonal Riesz transform.
Originally, the quasi-geostrophic equation is a nonlinear hyperbolic-type equation describing the geostrophic wind caused by Coriolis force around a low-pressure.
One with dissipation was introduced by Constantin, Majda and Tabak \cite{Cnstntn-Mjd-Tbk} as a model for the distribution of sea-surface potential temperature.
Namely, $\theta = \theta (t,x) \in \mathbb{R}$ stands for the thermal potential.
Several works treated this equation with the fractional laplacian $(-\Delta)^{\alpha/2}$ instead of $-\Delta$ under the abstract theory since $\alpha = 1$ is critical and the quasi-geostrophic equation resemble the three-dimensional incompressible Navier--Stokes equations from the perspective of their scale structure (see \cite{Cnstntn-Crdb-W,IwbchOkzk,Kslv-Nzrv-Vlbrg,Mur} and references there in).
However, in the context of large-time behavior of solutions, the nonlinear effects tend to be relatively weaker when $\alpha$ is small, and the critical is $\alpha = 3$.
We will discuss technical reasons for focusing on $\alpha = 2$ at the end of this section, but in both senses, our model is easier to work.
Therefore, even if the initial-data $\theta_0$ is large, unique solution $\theta$ exists globally in time and satisfies that
\begin{equation}\label{decayth}
	\| \theta (t) \|_{L^q (\mathbb{R}^2)}
	\le
	C (1+t)^{-\gamma_q}
\end{equation}
and
\begin{equation}\label{asymp0}
	\| \theta (t) - M_0 G(t) \|_{L^q (\mathbb{R}^2)} = o (t^{-\gamma_q})
\end{equation}
as $t \to +\infty$ for $1 \le q \le \infty$ and $\gamma_q = 1-1/q$ when $\theta_0 \in L^1 (\mathbb{R}^2) \cap L^\infty (\mathbb{R}^2)$, where $M_0 = \int_{\mathbb{R}^2} \theta_0 (x) dx = \int_{\mathbb{R}^2} \theta (t,x) dx$ is the conserved quantity and $G(t,x) = (4\pi t)^{-1} e^{-|x|^2/4t}$ is the fundamental solution (see\cite{Schnbk2}).
They are confirmed through Moser--Nash' energy theory and $L^p$-$L^q$ estimates for the heat semigroup.
For the same reason, the weighted estimate
\begin{equation}\label{decay-wt}
	\| |x|^m \theta (t) \|_{L^q (\mathbb{R}^2)} \le Ct^{-\gamma_q} (1+t)^{m/2}
\end{equation}
holds upon the condition $|x|^m \theta_0 \in L^1 (\mathbb{R}^2)$.
The asymptotic estimate \eqref{asymp0} means that the solution converges asymptotically to a radially symmetric function.
On the other hand, since it acts perpendicular to the diffusion, the Coriolis force does not affect the diffusion of a radially symmetric potential.
These facts are crucial for interpreting our main result.
Another important property of the top term $\Theta_0 (t) = M_0 G(t)$ is the parabolic scale that $\lambda^{2} \Theta_0 (\lambda^2 t, \lambda x) = \Theta_0 (t,x)$ for $\lambda > 0$.
Such a profile is determined uniquely.
Indeed, if there is some other profile $\bar\Theta_0$ which has the same scale, then this fulfills that $\| (\Theta_0-\bar\Theta_0) (t) \|_{L^q (\mathbb{R}^2)} = t^{-\gamma_q} \| (\Theta_0-\bar\Theta_0)(1) \|_{L^q (\mathbb{R}^2)}$ for $t > 0$ and $\| (\Theta_0-\bar\Theta_0) (t) \|_{L^q (\mathbb{R}^2)} \le \| (u-\Theta_0) (t) \|_{L^q (\mathbb{R}^2)} +   \| (u-\bar\Theta_0) (t) \|_{L^q (\mathbb{R}^2)} = o (t^{-\gamma_q})$ as $t \to +\infty$.
They contradict whenever $\bar\Theta_0 \neq \Theta_0$.
Generally, including higher-order expansions, asymptotic profiles based on the scales are fixed uniquely.
The unique higher-order expansion was derived in the preceding work as follows.
\begin{proposition}[cf.\cite{Ymmt26}]\label{prop-qg}
Let $\theta_0 \in L^1 (\mathbb{R}^2) \cap L^\infty (\mathbb{R}^2)$ and $x \theta_0 \in L^1 (\mathbb{R}^2)$.
Then the solution $\theta$ of \eqref{qg} fulfills that
\[
	\| (\theta - M_0 G - M_1 \cdot \nabla G)(t) \|_{L^q (\mathbb{R}^2)}
	= o (t^{-\gamma_q-1/2})
\]
as $t \to +\infty$ for $1 \le q \le \infty$ and $\gamma_q = 1-1/q$, where $M_0 = \int_{\mathbb{R}^2} \theta_0 (x) dx \in \mathbb{R}$ and $M_1 = - \int_{\mathbb{R}^2} x \theta_0 (x) dx \in \mathbb{R}^2$.
In addition, if $|x|^2 \theta_0 \in L^1 (\mathbb{R}^2)$, then the left-hand side is estimated by $O (t^{-\gamma_q-1}\log t)$ as $t\to\infty$.
\end{proposition}
The Escobedo--Zuazua theory \cite{Escbd-ZZ} together with the renormalization, which provides a measure of nonlinearity, predicted that there are some $K_1$ and $J_1$ satisfying $\| (K_1,J_1) (t) \|_{L^q(\mathbb{R}^2)} = t^{-\gamma_q-1/2} \| (K_1,$ $J_1) (1) \|_{L^q(\mathbb{R}^2)}$ for $t>0$ and $1\le q\le \infty$,
and $\theta \sim M_0 G + K_1 (t) \log t + M_1 \cdot \nabla G + J_1$ as $t \to +\infty$.
Particularly, it is expected that nonlinear characteristics will be reflected in $J_1$.
Such a fact can be observed in several equations that share the same scale structure as our model.
For the details, see the instructions for \eqref{bg} and \eqref{qgp} below.
It was considered that solutions to \eqref{qg} behave in a similar manner.
In fact, $K_1$ is a pseudo-logarithm and the nonlinearity $J_1$ never be seen.
For the reason why these components do not appear, contrary to expectations, see the proof of Proposition \ref{prop-asympmid-wt}.
Here, if we put $\Theta_1 (t) = M_1 \cdot \nabla G(t)$, then $\lambda^{2+1} \Theta_1 (\lambda^2 t, \lambda x) = \Theta_1 (t,x)$ for $\lambda > 0$ is fulfilled.
These $\Theta$s in Proposition \ref{prop-qg} are linear profiles of the solution since $\partial_t \Theta_m = \Delta \Theta_m$ are satisfied.
Needless to say, our equation is not a linear problem.
The effect of nonlinearity on the solution is small but certainly present.
In fact, as stated in the following main theorem, nonlinear characteristics appear in components two orders of magnitude smaller than the above-mentioned estimate.
To state it, we firstly introduce the lower-order profiles as
\[
	\Theta_0 (t) = M_0 G(t),\quad
	\Theta_1 (t) = M_1 \cdot \nabla G(t)
	\quad\text{and}\quad
	\Theta_2 (t) = \nabla \cdot (M_2 \nabla) G(t)
\]
for $M_0 = \int_{\mathbb{R}^2} \theta_0 (x) dx,~ M_1 = (M_1^j)$ for $M_1^j = - \int_{\mathbb{R}^2} x_j \theta_0 (x) dx$, and $M_2 = (M_2^{ij})$ for
\[
\begin{aligned}
	&M_2^{11} = \frac12 \int_{\mathbb{R}^2} x_1^2 \theta_0 (x) dx - \int_0^\infty \int_{\mathbb{R}^2} x_1 (\theta\mathcal{R}_2\theta- \Theta_0 \mathcal{R}_2\Theta_0) (t,x) dxdt,\\
	&M_2^{12} = \int_{\mathbb{R}^2} x_1 x_2 \theta_0 (x) dx + \int_0^\infty \int_{\mathbb{R}^2} x_1 (\theta\mathcal{R}_1\theta- \Theta_0 \mathcal{R}_1\Theta_0) (t,x) dxdt,\\
	&M_2^{21} = \int_{\mathbb{R}^2} x_2 x_1 \theta_0 (x) dx - \int_0^\infty \int_{\mathbb{R}^2} x_2 (\theta\mathcal{R}_2\theta- \Theta_0 \mathcal{R}_2\Theta_0) (t,x) dxdt,\\
	&M_2^{22} = \frac12 \int_{\mathbb{R}^2} x_2^2 \theta_0 (x) dx + \int_0^\infty \int_{\mathbb{R}^2} x_2 (\theta\mathcal{R}_1\theta- \Theta_0 \mathcal{R}_1\Theta_0) (t,x) dxdt.
\end{aligned}
\]
The higher-order profile of the linear dissipation is
\[
\begin{aligned}
	&\Theta_3 (t) = -\sum_{|\alpha|=3} \frac{\nabla^\alpha G(t)}{\alpha!} \int_{\mathbb{R}^2} y^\alpha \theta_0 (y) dy\\
	& - \sum_{2l+|\beta|=2} \frac{\partial_t^l \nabla^\beta \nabla G(t)}{\beta!} \cdot \int_0^\infty \int_{\mathbb{R}^2} (-s)^l y^\beta \left( \theta\mathcal{R}^\bot\theta - \Theta_0\mathcal{R}^\bot\Theta_0 - \Theta_0\mathcal{R}^\bot\Theta_1 - \Theta_1\mathcal{R}^\bot\Theta_0 \right) (s,y) dyds.
\end{aligned}
\]
They still fulfill $\partial_t \Theta_m = \Delta \Theta_m$.
The top term of nonlinearity is decided uniquely as
\begin{equation}\label{J3pre}
	J_3 (t) = J_3^{\mathrm{L}} (t) + J_3^{\mathrm{H}} (t)
\end{equation}
for
\[
\begin{aligned}
	&J_3^{\mathrm{L}} (t) = - \frac{\sqrt{2\pi} \left( 2M_0 (M_2^{22} - M_2^{11}) + (M_1^1)^2 - (M_1^2)^2 \right)}{128\pi} t^{-1/2} \partial_1 \partial_2 G(t)\\
	&- \frac{\sqrt{2\pi}\left( M_0 (M_2^{12} + M_2^{21}) - M_1^1 M_1^2\right)}{128\pi} t^{-1/2} (\partial_1^2 - \partial_2^2) G(t)
\end{aligned}
\]
and
\[
\begin{aligned}
	&J_3^{\mathrm{H}} (t)
	=\frac{\sqrt{\pi} \left( 2M_0 (M_2^{22} - M_2^{11}) + (M_1^1)^2 - (M_1^2)^2\right)}{16\pi^2}\\
		&\qquad \times 
        \int_0^t \int_0^\infty \int_0^\infty
		\lambda^{-3/2} e^{-\frac1{4\lambda}} \frac{\sigma^2\lambda}{(2s+\sigma^2\lambda)^3}\partial_1 \partial_2 \left( G \left( t-\frac{s^2}{2s+\sigma^2\lambda} \right) - G (t) \right) d\lambda d\sigma ds\\
	&+ \frac{\sqrt{\pi} \left( M_0 (M_2^{12} + M_2^{21}) - M_1^1 M_1^2\right)}{16\pi^2}\\
		&\qquad \times \int_0^t \int_0^\infty \int_0^\infty
		\lambda^{-3/2} e^{-\frac1{4\lambda}} \frac{\sigma^2\lambda}{(2s+\sigma^2\lambda)^3}(\partial_1^2 - \partial_2^2) \left( G \left( t-\frac{s^2}{2s+\sigma^2\lambda} \right) - G (t) \right) d\lambda d\sigma ds.
\end{aligned}
\]
Here, $J_3^{\mathrm{H}}$ is the component with a higher frequency than $J_3^{\mathrm{L}}$.
They represent the minute but largest and complex nonlinear flow contained within the solution.
Namely, $J_3$ reflects the effects of rotation caused by the Coriolis force.
For the original form and the origin of this profile, see \eqref{J3L} and \eqref{J3H}, and Remark \ref{rem33}, respectively.
Using these profiles, we establish our main results as follows.
\begin{theorem}\label{thm}
	Let $\theta_0 \in L^1 (\mathbb{R}^2) \cap L^\infty (\mathbb{R}^2)$ and $|x|^3 \theta_0 \in L^1 (\mathbb{R}^2)$.
	Then the solution $\theta$ of \eqref{qg} satisfies that
	\[
		\| (\theta - \Theta_0 - \Theta_1 - \Theta_2 - \Theta_3 - J_3) (t) \|_{L^q (\mathbb{R}^2)}
		= o (t^{-\gamma_q-3/2})
	\]
	as $t \to +\infty$ for $1 \le q \le \infty$ and $\gamma_q = 1-1/q$, where $\Theta_m$ and $J_3$ are defined as above.
	In addition, if $|x|^4 \theta_0 \in L^1 (\mathbb{R}^2)$, then the left-hand side is estimated by $O(t^{-\gamma_q-2} \log t)$ as $t \to +\infty$.
\end{theorem}
We firstly remark that
\[
	\lambda^{2+3} (\Theta_3,J_3) (\lambda^2 t, \lambda x) = (\Theta_3,J_3) (t,x)
\]
for $\lambda > 0$.
Hence, all components of $J_3$ decay with the same rate as $\Theta_3$ that is
\begin{equation}\label{opdecay}
	\| (\Theta_3,J_3) (t) \|_{L^q (\mathbb{R}^2)} = t^{-\gamma_q-3/2} \| (\Theta_3,J_3) (1) \|_{L^q (\mathbb{R}^2)}
\end{equation}
for $t > 0$.
The linear components $\Theta_m$ fulfill that $\int_{\mathbb{R}^2} x^\beta \Theta_m dx = 0$ for $|\beta| \le m-1$ and then
\[
	\| e^{t\Delta} \Theta_m (t_0) \|_{L^q (\mathbb{R}^2)} = O (t^{-\gamma_q-m/2})
\]
as $t \to +\infty$ for an arbitrary fixed $t_0 > 0$.
This is same rate as the optimal decay such as \eqref{opdecay} and, then, these behaviors are linear.
Needless to say that this decay also is obtained directly from the convolution.
Oppositely, since $\int_{\mathbb{R}^2} x^\beta J_3 dx \neq 0$ for some $|\beta| = 2$, the decay
\[
	\| e^{t\Delta} J_3 (t_0) \|_{L^q (\mathbb{R}^2)} = O (t^{-\gamma_q-1})
\]
as $t \to +\infty$ is optimal, but there appears to be a gap between this and \eqref{opdecay}.
More precisely, among the components of this profile, the spatial structure of $J_3^{\mathrm{L}}$ differs significantly from that of $J_3^{\mathrm{H}}$.
Namely, $\| e^{t\Delta} J_3^{\mathrm{L}} (t_0) \|_{L^q (\mathbb{R}^2)} = O (t^{-\gamma_q-1})$ and $\| e^{t\Delta} J_3^{\mathrm{H}} (t_0) \|_{L^q (\mathbb{R}^2)} = O (t^{-\gamma_q-2})$ as $t \to +\infty$ hold.
In other words, $J_3^{\mathrm{L}}$ has the same order of vanishing moments as $\partial_i \partial_j G$, but decays more rapidly.
While $J_3^{\mathrm{H}}$ has higher-orders but its decay-rate is slower than expected.
In general, for linear phenomena such as $\Theta_m$, there is a clear correspondence between scale and spatial structure.
These $J_3$s yield unique profile of the nonlinearity and are due to the nonlinear transport caused by the Coriolis force.
The potential temperature contains a mixture of components whose averaging is accelerated by the rotations and those whose averaging is inhibited by them.
The above discussion becomes clearer when we rewrite $J_3$.
Since $t-\frac{s^2}{2s+\sigma^2\lambda} \ge t/2 \gg 0$ for $(s,\sigma,\lambda) \in (0,t) \times \mathbb{R}_+ \times \mathbb{R}_+$, $G(t-\frac{s^2}{2s+\sigma^2\lambda}) - G(t)$ on $J_3^{\mathrm{H}}$ could be expanded around $- \frac{s^2}{2s+\sigma^2\lambda} = 0$ and then we see
\begin{equation}\label{J3}
\begin{aligned}
	&J_3 (t) = \frac{\sqrt{\pi} \left( 2M_0 (M_2^{22} - M_2^{11}) + (M_1^1)^2 - (M_1^2)^2\right)}{16\pi}
    \sum_{k=0}^\infty \frac{(2k+1)!!}{(2k-1)(2k+4)!!} \left( \frac{t}2 \right)^{k-\frac12} \frac{\partial_1 \partial_2 (-\Delta)^k G(t)}{k!}\\
	&+
	\frac{\sqrt{\pi} \left( M_0 (M_2^{12} + M_2^{21}) - M_1^1 M_1^2\right)}{16\pi}
    \sum_{k=0}^\infty\frac{(2k+1)!!}{(2k-1)(2k+4)!!} \left( \frac{t}2 \right)^{k-\frac12} \frac{(\partial_1^2-\partial_2^2) (-\Delta)^k G(t)}{k!}.
\end{aligned}
\end{equation}
The terms of $k=0$ are $J_3^{\mathrm{L}}$ itself and $\sum_{k=1}^\infty$ is coming from $J_3^{\mathrm{H}}$.
This background explains why readers find the form of the coefficient unsettling, i.e., $\frac{(2k+1)!!}{(2k-1)(2k+4)!!}$ is negative only for $k = 0$.
We confirm that this summation converges in $L^q (\mathbb{R}^2)$ for $1 \le q \le \infty$.
Comparing with the expansion of $G(t/2)$ for $t > 0$ as
\begin{equation}\label{Ghf}
	\left( \frac{t}2 \right)^{-1/2} \partial_i \partial_j G\left( \frac{t}2 \right)
	= - \sum_{k = 0}^\infty \left( \frac{t}2 \right)^{k-1/2} \frac{\partial_i \partial_j (-\Delta)^k G(t)}{k!},
\end{equation}
then this converges absolutely in the Lebesgue spaces.
Far from being worse, $J_3$ contains auxiliary factor $\frac{(2k+1)!!}{(2k-1)(2k+4)!!} \approx 1/k^2$.
Comparing with \eqref{Ghf} also suggests that, among the components in \eqref{J3}, the terms of low $k$ closely reflect the nonlinear structure of \eqref{qg}.
The coefficients of $\Theta_2$ and $\Theta_3$ contain unclear moments.
They arise sure from the nonlinear term but
are linear components since $\partial_t \Theta_m = \Delta \Theta_m$ still holds.
Reader may concern that the second moment $M_2$ diverges to infinity since \eqref{decay-wt} and Proposition \ref{prop-qg} suggest that $\int_{\mathbb{R}^2} x_i (\theta\mathcal{R}_j\theta - \Theta_0\mathcal{R}_j\Theta_0) dx = O (t^{-1}\log t)$ as $t \to +\infty$.
The third moments on $\Theta_3$ contain the same concern.
We will solve this crux in Sections \ref{sect2} and \ref{sect3}.
Some readers may associate the coefficient of first parts of $J_3$ with the special case of isotropy of incompressible Navier--Stokes flow (cf.\cite{Brndls-Okb,MykwSchnbk}).
This case is associated to $M_1^1 = M_1^2$ and $M_2^{11} = M_2^{22}$.
We recall here Proposition \ref{prop-qg}, since $\Theta_2$ has no logarithmic shift, and confirm that the logarithmic evolution on sharp estimate of this proposition is false, i.e., \[
	\| (\theta - M_0 G - M_1 \cdot \nabla G)(t) \|_{L^q (\mathbb{R}^2)} = O (t^{-\gamma_q-1})
\]
as $t \to +\infty$ holds in fact.
We will remove such false logarithms in several scene.
However, we expect that the logarithmic evolution on second assertion of Theorem \ref{thm} is optimal (see Remark \ref{rem31}).

Comparing with similar equations, Theorem \ref{thm} says that the nonlinearity of \eqref{qg} is weaker than what is expected from its scale.
Here, we compare with the convection-diffusion equation that
\begin{equation}\label{bg}
\left\{
\begin{array}{lr}
	\partial_t \rho - \Delta \rho = a\cdot\nabla (|\rho|\rho),
	&
	t > 0,~ x \in \mathbb{R}^2,\\
	\rho (0,x) = \rho_0 (x),
	&
	x \in \mathbb{R}^2,
\end{array}
\right.
\end{equation}
and the artificial equation that
\begin{equation}\label{qgp}
\left\{
\begin{array}{lr}
	\partial_t \vartheta - \nabla \cdot (\vartheta\mathcal{R}\vartheta) = \Delta \vartheta,
	&
	t > 0,~ x \in \mathbb{R}^2,\\
	\vartheta (0,x) = \vartheta_0 (x),
	&
	x \in \mathbb{R}^2,
\end{array}
\right.
\end{equation}
where $a \in \mathbb{R}^2$ is a given direction and $\mathcal{R} = (\mathcal{R}_1, \mathcal{R}_2)$ is the forward-directed Riesz transform.
These three equations are identical not only in scale but also in the order of differentiation if the Riesz transforms are regarded as zeroth-order derivatives.
The only difference lies in their symmetry.
For these problems,
\begin{equation}\label{expcd}
	\rho (t)  \sim M_0 G(t) + \frac{|M_0| M_0}{8\pi} a \cdot \nabla G(t) \log t + M_1 \cdot \nabla G(t) + J_1^\rho (t)
\end{equation}
for $M_0 = \int_{\mathbb{R}^2} \rho_0 (x) dx,~ M_1 = - \int_{\mathbb{R}^2} x \rho_0 (x) dx + a \int_0^\infty \int_{\mathbb{R}^2} ((|\rho|\rho) (t,x) - |M_0|M_0 G^2 (1+t,x)) dxdt$ and
\[
	J_1^\rho (t) = \frac{|M_0|M_0}{8\pi} \int_0^t s^{-1} (a\cdot\nabla G(t-\tfrac{s}2) - a\cdot\nabla G(t)) ds,
\]
and
\[
	\vartheta (t) \sim M_0 G(t) + M_1 \cdot \nabla G(t) + J_1^\vartheta (t)
\]
for $M_0 = \int_{\mathbb{R}^2} \vartheta_0 (x) dx,~ M_1 = - \int_{\mathbb{R}^2} x \vartheta_0 (x) dx$ and
\begin{equation}\label{Jart}
	J_1^\vartheta (t) = \frac{\sqrt{\pi}M_0^2}{8\pi^2} \int_0^t \int_0^\infty \int_0^\infty \lambda^{-3/2} e^{-\frac1{4\lambda}} s (2s+\sigma^2\lambda)^{-2} \Delta G\left( t-\tfrac{s^2}{2s+\sigma^2 \lambda}\right) d\lambda d\sigma ds
\end{equation}
as $t \to +\infty$ are hold (see also \cite{Ymmt26}).
For \eqref{bg}, the expansion up to the logarithmic shift $K(t) = K_1 (t) \log t$ for $K_1 = \frac{|M_0|M_0}{8\pi} a \cdot \nabla G(t)$ is well-known (see\cite{Escbd-ZZ,Ksb,ZZ93}).
Such a shift arises from the nonlinear term and satisfies the linear equation $\partial_t K = \Delta K + \frac{K}{t\log t}$.
For these profiles
$
	\lambda^{2+1} (M_1 \cdot \nabla G, K_1, J_1^\rho, J_1^\vartheta) (\lambda^2 t, \lambda x) = (M_1 \cdot \nabla G, K_1, J_1^\rho, J_1^\vartheta) (t,x)
$
holds for $\lambda > 0$ and, then, their decay rates are same that
\[
    \| (M_1 \cdot \nabla G, K_1, J_1^\rho, J_1^\vartheta)(t) \|_{L^q (\mathbb{R}^2)}
    =
    t^{-\gamma_q-1/2} \| (M_1 \cdot \nabla G, K_1, J_1^\rho, J_1^\vartheta)(1) \|_{L^q (\mathbb{R}^2)}
\]
for $t>0$.
On the other hand, both of $J_1^\rho$ and $J_1^\vartheta$ have the higher-order derivatives than one of $(M_1 \cdot \nabla G,K_1)$.
In particular,
\[
    \| e^{t\Delta} J_1^\rho (t_0) \|_{L^q (\mathbb{R}^2)} = O(t^{-\gamma_q-3/2})\quad\text{and}\quad
    \| e^{t\Delta} J_1^\vartheta (t_0) \|_{L^q (\mathbb{R}^2)} = O(t^{-\gamma_q-1})
\]
as $t\to+\infty$ hold, and then, the correspondence between the scale and the spatial structure is collapse and the effect of nonlinear-distorting of each solution is expressed in early stage.
As another example, solutions to the three-dimensional Nernst--Planck drift-diffusion equation, which has the same scale as above three models, also exhibit nonlinear profile similar to $J_1^\vartheta$ (cf.\cite{Ymmt26ZAMP}).
But, judging from Theorem \ref{thm}, our problem \eqref{qg} does not have such a strong nonlinearity.
This is because the radially symmetrization \eqref{asymp0} weakens the nonlinearity arising from the Coriolis force as \eqref{begin} below.
This corresponds precisely to replacing $-\Delta G= \nabla \cdot \mathcal{R} (-\Delta)^{1/2} G$ with $0 = \nabla \cdot \mathcal{R}^\bot (-\Delta)^{1/2} G$ in \eqref{Jart}.

Before closing this section, we recall our model with dissipation $(-\Delta)^{\alpha/2}$.
One of the technical reasons we consider the case $\alpha = 2$ is that the fundamental solution of the linear semigroup $e^{t\Delta}$ is given by the specific elementary function $G(t)$.
This is advantageous when writing nonlinear profile $J_3$ explicitly.
The other case $\alpha = 1$ sure provides the concrete fundamental solution.
But, it is important to note that, in the case $\alpha = 2$ only, the fundamental solution is rapidly decreasing with respect to space.
If we treat small $\alpha$, then the nonlinear effects become relatively weak.
So we should decide higher-order expansion to get nonlinear profiles such as $J_3$.
In this procedure, integrability of moments of solutions with high-order is required.
In other words, weighted estimates such as \eqref{decay-wt} with large $m$ are needed.
The slow decay of solutions may hinder this integrability.
To address the fractional dissipation, it will be necessary to develop a framework separate from the one presented in this paper (see for example \cite{Brndls-Krch}).
In fact, for orders up to the second, there is preceding work that deals with expansion of solutions to the fractional equations.
In \cite{OhAdhIwbch}, the expansion up to the second order that includes the term arising from the nonlinear effect is derived.
This term actually consists of the linear profiles such as the part of $\Theta_2$ and some errors and phantoms.

\vspace{2mm}

\noindent
\textbf{Notations.}
We often omit the spatial variable from a function, for example, $\theta (t) = \theta (t,x)$.
This variable is hidden also in convolutions.
Precisely, $e^{t\Delta} \theta_0 = G(t) * \theta_0 = \int_{\mathbb{R}^2} G(t,x-y) \theta_0 (y) dy$ and $\int_0^t g(t-s) * f(s) ds = \int_0^t \int_{\mathbb{R}^2} g(t-s,x-y) f(s,y) dyds$.
The Fourier transform and its inverse are defined by $\mathcal{F} [\varphi] (\xi) = (2\pi)^{-1}$ $\int_{\mathbb{R}^2} e^{-ix\cdot\xi} \varphi (x) dx$ and $\mathcal{F}^{-1} [\varphi] (x) = (2\pi)^{-1} \int_{\mathbb{R}^2} e^{ix\cdot\xi} \varphi (\xi) d\xi$, where $i = \sqrt{-1}$.
The derivations are denoted by $\partial_t = \partial/\partial t,~ \partial_j = \partial/\partial x_j$ for $j = 1,2$, and $\nabla = (\partial_1,\partial_2),~ \nabla^\bot = (-\partial_2,\partial_1)$ and $\Delta = \partial_1^2 + \partial_2^2$.
The Riesz transforms $\mathcal{R}_j \varphi = \partial_j (-\Delta)^{-1/2} \varphi = \mathcal{F}^{-1}[ i\xi_j/|\xi| \mathcal{F} [\varphi]]$ for $j=1,2$, and $\mathcal{R}^\bot = (-\mathcal{R}_2,\mathcal{R}_1)$ and $\mathcal{R} = (\mathcal{R}_1,\mathcal{R}_2)$ are regarded as zeroth-order derivatives.
The Lebesgue space and its norm are denoted by $L^q (\mathbb{R}^2)$ and $\| \cdot \|_{L^q (\mathbb{R}^2)}$, that is, $\| f \|_{L^q (\mathbb{R}^2)} = (\int_{\mathbb{R}^2} |f(x)|^q dx)^{1/q}$ for $1 \le q < \infty$ and $\| f \|_{L^\infty (\mathbb{R}^2)}$ is the essential supremum.
The inner-product of $f$ and $g \in L^2 (\mathbb{R}^2)$ is denoted by $\langle f,g \rangle = \int_{\mathbb{R}^2} (f\overline{g}) (x) dx$.
The heat kernel and its decay rate on $L^q (\mathbb{R}^2)$ are symbolized by $G(t,x) = (4\pi t)^{-1} e^{-|x|^2/(4t)}$ and $\gamma_q = 1-1/q$.
The indices of functions correspond to their scales or decay-rates, for example, $\lambda^{2+m} \Theta_m (\lambda^2 t, \lambda x) = \Theta_m (t,x)$ for $\lambda > 0,~ \| J_3 (t) \|_{L^q (\mathbb{R}^2)} = t^{-\gamma_q-3/2} \| J_3 (1) \|_{L^q (\mathbb{R}^2)}$ for $t > 0$, and $\| r_4 (t) \|_{L^q (\mathbb{R}^2)} = O (t^{-\gamma_q-2} \log t)$ as $t \to +\infty$.
We employ Landau symbol.
Namely, $f(t) = o(t^{-\mu})$ and $g(t) = O(t^{-\mu})$ mean $t^\mu f(t) \to 0$ and $t^\mu g(t) \to c$ for some $c \in \mathbb{R}$ such as $t \to +\infty$ or $t \to +0$, respectively.
Various nonnegative constants are denoted by $C$.

\section{Preliminaries}\label{sect2}
To prove our main result, we prepare some tools.
The following lemma, which we introduce to clarify the discussion, requires no proof.
\begin{lemma}\label{lem-discuss}
Let $g \in C((0,1), L^1 (\mathbb{R}^2))$ have $f,h \in L^1 (\mathbb{R}^2)$ and $E \in L^1 ((0,1)\times\mathbb{R}^2)$ such as $g(s) - s^{-3/2} f \in L^1 ((0,1)\times\mathbb{R}^2)$ and $g(s) = s^{-3/2} h + E(s)$.
Then $h = f$ holds.
\end{lemma}
%
%
It is well known that Riesz transform is bounded in $L^q (\mathbb{R}^2)$ for $1 < q < \infty$.
\begin{lemma}[see \cite{Stin}]\label{lemRz}
For $1 < q < \infty$, there is $C$ such that
$
	\| \mathcal{R}^\bot \varphi \|_{L^q (\mathbb{R}^2)}
	\le
	C \| \varphi \|_{L^q (\mathbb{R}^2)}
$
for $\varphi \in L^q (\mathbb{R}^2)$.
\end{lemma}
Needless to say that this lemma is available for $\mathcal{R}$ and $\mathcal{R}_j$.
For some weighted estimates of Riesz transform, the following H\"ormander--Mikhlin estimate is required.
\begin{lemma}[cf. \cite{Sbt-Smz}]\label{lemHM}
Let $N \in \mathbb{Z}_+,~ 0 < \mu \le 1$ and $\lambda = N+\mu-2$.
Assume that $\varphi \in C^\infty (\mathbb{R}^2\backslash \{ 0 \})$ satisfies the following conditions:
\begin{itemize}
\item
	$\nabla^\gamma \varphi \in L^1 (\mathbb{R}^2)$ for any $\gamma \in \mathbb{Z}_+^2$ with $|\gamma| \le N$;
\item
	$|\nabla^\gamma \varphi (\xi)| \le C_\gamma |\xi|^{\lambda-|\gamma|}$ for $\xi\neq 0$ and $\gamma \in \mathbb{Z}_+^2$ with $|\gamma| \le N+2$.
\end{itemize}
Then
\[
	\sup_{x\neq 0} (|x|^{2+\lambda} |\mathcal{F}^{-1} [\varphi] (x)|) < +\infty
\]
holds.
\end{lemma}
Moreover, Riesz transforms do not change the scale.
Namely, $\lambda^{2+m} \mathcal{R}^\bot \Theta_m (\lambda^2 t, \lambda x) = \mathcal{R}^\bot \Theta_m (t,x)$ for $\lambda > 0$.
Now the following bilinear estimate is available.
\begin{corollary}\label{cor-asymplow-wt}
Let $\theta_0 \in L^1 (\mathbb{R}^2) \cap L^\infty (\mathbb{R}^2)$ and $|x|^2 \theta_0 \in L^1 (\mathbb{R}^2)$.
Then
\[
	\| |x|^\mu (\theta \mathcal{R}^\bot \theta - \Theta_0 \mathcal{R}^\bot \Theta_0 - \Theta_1 \mathcal{R}^\bot \Theta_0 - \Theta_0 \mathcal{R}^\bot \Theta_1) (t) \|_{L^q (\mathbb{R}^2)}
	=
	O(t^{-\gamma_q-2+\mu/2} \log t)
\]
as $t \to +\infty$ holds for $1 \le q < \infty$ and $0 \le \mu \le 2$.
\end{corollary}
\begin{proof}
Applying \eqref{decay-wt} and Proposition \ref{prop-qg} to
\begin{equation}\label{bsbilin}
\begin{aligned}
	&\theta \mathcal{R}^\bot \theta - \Theta_0 \mathcal{R}^\bot \Theta_0 - \Theta_1 \mathcal{R}^\bot \Theta_0 - \Theta_0 \mathcal{R}^\bot \Theta_1\\
	&= \theta \mathcal{R}^\bot (\theta - \Theta_0 - \Theta_1) + (\theta - \Theta_0) \mathcal{R}^\bot \Theta_1 + (\theta-\Theta_0-\Theta_1) \mathcal{R}^\bot \Theta_0
\end{aligned}
\end{equation}
immediately completes the proof.
Here, Lemma \ref{lemHM} with $\varphi = \frac{\xi^\bot (M_1\cdot \xi)}{|\xi|} e^{-t|\xi|^2}$ and $\frac{\xi^\bot}{|\xi|} e^{-t|\xi|^2}$ were used for the second and last terms, respectively.
\end{proof}
We used the scales that $\| |x|^\mu \mathcal{R}^\bot \Theta_m (t) \|_{L^q (\mathbb{R}^2)} = t^{-\gamma_q-m/2 + \mu/2} \| |x|^\mu \mathcal{R}^\bot \Theta_m (1) \|_{L^q (\mathbb{R}^2)}$.
From Lemma \ref{lemHM}, we see $|x|^\mu \mathcal{R}^\bot \Theta_m \in L^\infty (\mathbb{R}^2)$ for $0 \le \mu \le 2$ since $\Theta_m \in \mathcal{S} (\mathbb{R}^2)$.
For $\mathcal{R}^\bot \theta$, we omitted the case $q = \infty$.
When this corollary is applied, the singularies as $t \to +0$ coming from $\Theta_1 \mathcal{R}^\bot \Theta_0, \Theta_0 \mathcal{R}^\bot \Theta_1$ and $\theta\mathcal{R}^\bot\theta$ should be remarked.

\vspace{3mm}

\noindent
\textbf{The phantoms.}
In this paper, we often encounter a `phantom' term.
For example, an $L^1$-function $\varphi$ satisfying $\varphi(-x) = - \varphi(x)$ fulfills that $\int_{\mathbb{R}^2} \varphi (x) dx = 0$.
More precisely,
\begin{equation}\label{phantom1}
	\int_{\mathbb{R}^2} x^\alpha \nabla^\beta G \mathcal{R}^\bot \nabla^\gamma G dx = 0
\end{equation}
when $|\alpha+\beta+\gamma|$ is even.
Such an integral could be vanishing even if $\varphi (-x) = \varphi (x)$.
For example, $\int_{\mathbb{R}^2} x_i (\partial_i \partial_j G \mathcal{R}_i G) dx = 0$ for $i \neq j$ since the integrand is odd in both of $x_i$ and $x_j$.
For other examples,
\begin{equation}\label{phantom2}
	\int_{\mathbb{R}^2} f \mathcal{R}^\bot f dx = \int_{\mathbb{R}^2} (f\mathcal{R}^\bot g + g\mathcal{R}^\bot f) dx = 0
\end{equation}
for $f,g \in L^2 (\mathbb{R}^2)$ since $\mathcal{R}^\bot$ is skew-adjoint, i.e., $\langle f,\mathcal{R}^\bot g \rangle = - \langle \mathcal{R}^\bot f, g \rangle$.
The most distinctive example for orthogonal operator is that $\nabla g * \nabla^\bot f$ and $\langle \nabla f, \nabla^\bot g \rangle$ are vanishing since $\nabla^\bot \cdot \nabla = 0$.
Due to spatial structures, such phantoms frequently appear when dealing with the quasi-geostrophic equation and we will see that they are useful to estimate some tricky quantities.
In particular, they will cancel out apparent singularities and improve the decay-rates.
We use them firstly to show the following proposition.

This proposition says that the logarithmic evolutions on Proposition \ref{prop-qg} and Corollary \ref{cor-asymplow-wt} are false.
Indeed, in this proposition, $\Theta_2$ has no logarithm.
\begin{proposition}\label{prop-asympmid-wt}
Let $\theta_0 \in L^1 (\mathbb{R}^2) \cap L^\infty (\mathbb{R}^2)$ and $|x|^3 \theta_0 \in L^1 (\mathbb{R}^2)$.
Then
\[
	\| |x|^\mu (\theta - \Theta_0 - \Theta_1 -\Theta_2)(t) \|_{L^q (\mathbb{R}^2)} \le C t^{-\gamma_q-1+\mu/2} (1+t)^{-1/2} \log (2+t)
\]
holds for $0 \le \mu \le 1$ and $1 \le q \le \infty$.
\end{proposition}
\begin{proof}
The proof is similar as one of \cite[Proposition 2.1]{Ymmt25NARWA}.
We convert \eqref{qg} to the mild solution as follows:
\begin{equation}\label{ms}
	\theta (t) = e^{t\Delta} \theta_0 - \int_0^t \nabla e^{(t-s)\Delta}\cdot (\theta\mathcal{R}^\bot\theta) (s) ds.
\end{equation}
Here, we used the relation $\mathcal{R}^\bot \theta \cdot \nabla \theta = \nabla \cdot (\theta \mathcal{R}^\bot \theta)$.
The treatment for the term of initial-data is well known:
\[
	\biggl\| |x|^\mu \biggl(
	e^{t\Delta} \theta_0 - \sum_{|\alpha| = 0}^2 \frac{\nabla^\alpha G(t)}{\alpha!} \int_{\mathbb{R}^2} (-y)^\alpha \theta_0 (y) dy
	\biggr) \biggr\|_{L^q (\mathbb{R}^2)}
	\le C t^{-\gamma_q-1+\mu/2} (1+t)^{-1/2}.
\]
We find several parts of $\Theta_0, \Theta_1$ and $\Theta_2$ from here.
The nonlinear term is converted as
\begin{equation}\label{begin}
\begin{aligned}
	&\int_0^t \nabla e^{(t-s)\Delta} \cdot (\theta\mathcal{R}^\bot\theta) (s) ds\\
	&=
	\int_0^t \int_{\mathbb{R}^2} \left( \nabla G(t-s,x-y) - \nabla G(t,x) \right) \cdot (\theta\mathcal{R}^\bot\theta - \Theta_0\mathcal{R}^\bot\Theta_0) (s,y) dyds.
\end{aligned}
\end{equation}
In other words, the integrand is smaller than it appears.
Here, we omitted some phantoms.
Namely, \eqref{phantom2} yields $\int_{\mathbb{R}^2} (\theta\mathcal{R}^\bot\theta - \Theta_0\mathcal{R}^\bot\Theta_0) dy = 0$.
Particularly, the logarithm $\int_{\mathbb{R}^2} (\Theta_0\mathcal{R}^\bot\Theta_0) (s,y) dy = s^{-1} \int_{\mathbb{R}^2}$ $(\Theta_0\mathcal{R}^\bot\Theta_0) (1,y) dy$ as in \eqref{expcd} does not appear.
Another phantom $\nabla e^{(t-s)\Delta}\cdot (\Theta_0\mathcal{R}^\bot\Theta_0) (s)$ is hidden since $\Theta_0$ is radially symmetric and then $\nabla\cdot (\Theta_0 \mathcal{R}^\bot \Theta_0) = 0$.
We further expand the nonlinear term through Escobedo--Zuazua theory together with the renormalization.
This coupling originally developed in \cite{KtM} for the Burgers equation.
Hence, applying Taylor expansion into \eqref{begin} yields that
\begin{equation}\label{bs1}
\begin{aligned}
	&-\int_0^t \nabla e^{(t-s)\Delta}\cdot (\theta\mathcal{R}^\bot\theta) (s) ds
	=
	\sum_{|\beta|=1} \nabla^\beta \nabla G(t) \cdot \int_0^t \int_{\mathbb{R}^2} y^\beta (\theta\mathcal{R}^\bot\theta - \Theta_0\mathcal{R}^\bot\Theta_0) (s,y) dyds\\
	&- \int_0^t \int_{\mathbb{R}^2} \biggl( \nabla G(t-s,x-y) - \sum_{|\beta|=0}^1 \nabla^\beta \nabla G(t,x) (-y)^\beta \biggr) \cdot (\theta\mathcal{R}^\bot\theta - \Theta_0\mathcal{R}^\bot\Theta_0) (s,y) dyds.
\end{aligned}
\end{equation}
The coefficient of first term is rewritten as
\[
\begin{aligned}
	&\int_0^t \int_{\mathbb{R}^2} y^\beta (\theta\mathcal{R}^\bot\theta - \Theta_0\mathcal{R}^\bot\Theta_0) (s,y) dyds\\
	&=
	\int_0^\infty \int_{\mathbb{R}^2} y^\beta (\theta\mathcal{R}^\bot\theta - \Theta_0\mathcal{R}^\bot\Theta_0)
	(s,y) dyds
	- \int_t^\infty \int_{\mathbb{R}^2} y^\beta \biggl( \theta\mathcal{R}^\bot\theta - \sum_{m_1+m_2=0}^1 \Theta_{m_1} \mathcal{R}^\bot\Theta_{m_2} \biggr) (s,y) dyds
\end{aligned}
\]
since \eqref{phantom1} says that $\int_{\mathbb{R}^2} y^\beta (\Theta_1 \mathcal{R}^\bot \Theta_0) dy = \int_{\mathbb{R}^2} y^\beta (\Theta_0 \mathcal{R}^\bot \Theta_1) dy = 0$ for $|\beta| = 1$, and then Corollary \ref{cor-asymplow-wt} guarantee that
\[
\begin{aligned}
	&\int_0^\infty \int_{\mathbb{R}^2} y^\beta (\theta\mathcal{R}^\bot\theta - \Theta_0\mathcal{R}^\bot\Theta_0) (s,y) dyds
	= \int_0^\infty \int_{\mathbb{R}^2} y^\beta \biggl( \theta\mathcal{R}^\bot\theta - \sum_{m_1+m_2=0}^1 \Theta_{m_1} \mathcal{R}^\bot\Theta_{m_2} \biggr) (s,y) dyds
\end{aligned}
\]
and $M_2$ converge.
The second term of \eqref{bs1} is further renormalized as
\[
\begin{aligned}
	&\int_0^t \int_{\mathbb{R}^2} \biggl( \nabla G(t-s,x-y) - \sum_{|\beta|=0}^1 \nabla^\beta \nabla G(t,x) (-y)^\beta \biggr) \cdot (\theta\mathcal{R}^\bot\theta - \Theta_0\mathcal{R}^\bot\Theta_0) (s,y) dyds\\
	&=\sum_{m_1+m_2=1} \int_0^t \int_{\mathbb{R}^2} \biggl( \nabla G(t-s,x-y) - \sum_{|\beta|=0}^1 \nabla^\beta \nabla G(t,x) (-y)^\beta \biggr) \cdot (\Theta_{m_1}\mathcal{R}^\bot\Theta_{m_2}) (s,y) dyds\\
	&+\int_0^t \int_{\mathbb{R}^2} \biggl( \nabla G(t-s,x-y) - \sum_{|\beta|=0}^1 \nabla^\beta \nabla G(t,x) (-y)^\beta \biggr)
	\cdot \biggl( \theta\mathcal{R}^\bot\theta - \sum_{m_1+m_2=0}^1 \Theta_{m_1} \mathcal{R}^\bot\Theta_{m_2} \biggr) (s,y) dyds.
\end{aligned}
\]
Here, the first part of this is vanishing.
Indeed, the phantoms are handled as above and
\eqref{phantom2} provides $\int_{\mathbb{R}^2} (\Theta_0\mathcal{R}^\bot\Theta_1 + \Theta_1\mathcal{R}^\bot\Theta_0) dy = 0$.
For another term, we see that
\[
\begin{aligned}
	&\nabla e^{(t-s)\Delta}\cdot  (\Theta_1\mathcal{R}^\bot\Theta_0+\Theta_0\mathcal{R}^\bot\Theta_1) (s)
	=
	M_0 e^{(t-s)\Delta}(M_1\cdot\nabla) \nabla \cdot (G\mathcal{R}^\bot G) (s) = 0.
\end{aligned}
\]
Applying them into \eqref{bs1} yields that
\[
\begin{aligned}
	&-\int_0^t \nabla e^{(t-s)\Delta}\cdot  (\theta\mathcal{R}^\bot\theta) (s) ds
	=
	\sum_{|\beta|=1} \nabla^\beta \nabla G(t) \cdot \int_0^\infty \int_{\mathbb{R}^2} y^\beta ( \theta\mathcal{R}^\bot\theta - \Theta_0\mathcal{R}^\bot\Theta_0 ) (s,y) dyds + r_3(t)
\end{aligned}
\]
for
\begin{equation}\label{r3}
\begin{aligned}
	&r_3 (t) = -\sum_{|\beta|=1} \nabla^\beta \nabla G(t) \cdot \int_t^\infty \int_{\mathbb{R}^2} y^\beta \biggl( \theta\mathcal{R}^\bot\theta - \sum_{m_1+m_2=0}^1 \Theta_{m_1}\mathcal{R}^\bot\Theta_{m_2} \biggr) (s,y) dyds\\
	&-\int_0^t \int_{\mathbb{R}^2} \biggl( \nabla G(t-s,x-y) - \sum_{|\beta|=0}^1 \nabla^\beta \nabla G(t,x) (-y)^\beta \biggr)
    \cdot \biggl( \theta\mathcal{R}^\bot\theta - \sum_{m_1+m_2=0}^1 \Theta_{m_1}\mathcal{R}^\bot\Theta_{m_2} \biggr) (s,y) dyds.
\end{aligned}
\end{equation}
Here, we complete the coefficient $M_2$ of $\Theta_2$ and see that $\theta = \Theta_0 + \Theta_1 + \Theta_2 + r_3^0 + r_3$ for
\[
	r_3^0 (t) = \int_{\mathbb{R}^2} \biggl( G(t,x-y) - \sum_{|\alpha| = 0}^2 \frac{\nabla^\alpha G(t)}{\alpha!} (-y)^\alpha \biggr) \theta_0 (y) dy.
\]
We show below that $\| r_3 (t) \|_{L^q (\mathbb{R}^2)} = O (t^{-\gamma_q-3/2} (\log t)^2)$ as $t \to + \infty$ for $1 \le q \le \infty$ from Taylor theorem with Corollary \ref{cor-asymplow-wt}.
Then, at this point, we see that
\[
	\| (\theta - \Theta_0 - \Theta_1 - \Theta_2) (t) \|_{L^q (\mathbb{R}^2)} = O (t^{-\gamma_q-3/2} (\log t)^2)
\]
as $t \to + \infty$.
Indeed, the coefficient of the first part of $r_3$ is estimated as
\[
	\biggl| \int_t^\infty \int_{\mathbb{R}^2} y^\beta \biggl( \theta\mathcal{R}^\bot\theta - \sum_{m_1+m_2=0}^1 \Theta_{m_1} \mathcal{R}^\bot\Theta_{m_2} \biggr) (s,y) dyds \biggr|
	\le
	C \int_t^\infty s^{-3/2} \log s ds
	=
	O (t^{-1/2} \log t)
\]
as $t \to +\infty$.
The second part of $r_3$ is separated and converted as
\[
\begin{aligned}
	&\int_0^t \int_{\mathbb{R}^2} \biggl( \nabla G(t-s,x-y) - \sum_{|\beta|=0}^1 \nabla^\beta \nabla G(t,x) (-y)^\beta \biggr)
	\cdot \biggl( \theta\mathcal{R}^\bot\theta - \sum_{m_1+m_2=0}^1 \Theta_{m_1}\mathcal{R}^\bot\Theta_{m_2} \biggr) (s,y) dyds\\
	&=\int_0^{t/2} \int_{\mathbb{R}^2} ( \nabla G(t-s,x-y) - \nabla G(t,x-y) )
	\cdot \biggl( \theta\mathcal{R}^\bot\theta - \sum_{m_1+m_2=0}^1 \Theta_{m_1}\mathcal{R}^\bot\Theta_{m_2} \biggr) (s,y) dyds\\
	&+ \int_0^t \int_{\mathbb{R}^2} \biggl( \nabla G(t,x-y) - \sum_{|\beta|=0}^1 \nabla^\beta \nabla G(t,x) (-y)^\beta \biggr)
	\cdot \biggl( \theta\mathcal{R}^\bot\theta - \sum_{m_1+m_2=0}^1 \Theta_{m_1}\mathcal{R}^\bot\Theta_{m_2} \biggr) (s,y) dyds\\
	&+\int_{t/2}^t \int_{\mathbb{R}^2} ( \nabla G(t-s,x-y) - \nabla G(t,x-y) )
	\cdot \biggl( \theta\mathcal{R}^\bot\theta - \sum_{m_1+m_2=0}^1 \Theta_{m_1}\mathcal{R}^\bot\Theta_{m_2} \biggr) (s,y) dyds\\
	&= \int_0^{t/2} \int_{\mathbb{R}^2} \int_0^1 \partial_t \nabla G(t-\lambda s, x-y) \cdot (-s)\biggl( \theta\mathcal{R}^\bot\theta - \sum_{m_1+m_2=0}^1 \Theta_{m_1}\mathcal{R}^\bot\Theta_{m_2} \biggr) (s,y) d\lambda dyds\\
	&+ \sum_{|\beta|=2} \int_0^t \int_{\mathbb{R}^2} \int_0^1 \frac{\nabla^\beta \nabla G(t,x-\lambda y)}{\beta!} (1-\lambda) \cdot (-y)^\beta \biggl( \theta\mathcal{R}^\bot\theta - \sum_{m_1+m_2=0}^1 \Theta_{m_1}\mathcal{R}^\bot\Theta_{m_2} \biggr) (s,y) d\lambda dyds\\
	&+\int_{t/2}^t \int_{\mathbb{R}^2} ( \nabla G(t-s,x-y) - \nabla G(t,x-y) )
	\cdot \biggl( \theta\mathcal{R}^\bot\theta - \sum_{m_1+m_2=0}^1 \Theta_{m_1}\mathcal{R}^\bot\Theta_{m_2} \biggr) (s,y) dyds.
\end{aligned}
\]
Therefore, Hausdorff--Young inequality together with Corollary \ref{cor-asymplow-wt} provide that
\[
\begin{aligned}
	&\biggl\| \int_0^t \int_{\mathbb{R}^2} \biggl( \nabla G(t-s,x-y) - \sum_{|\beta|=0}^1 \nabla^\beta \nabla G(t,x) (-y)^\beta \biggr)\\
	 &\hspace{15mm}
	\cdot \biggl( \theta\mathcal{R}^\bot\theta - \sum_{m_1+m_2=0}^1 \Theta_{m_1}\mathcal{R}^\bot\Theta_{m_2} \biggr) (s,y) dyds \biggr\|_{L^q (\mathbb{R}^2)}\\
	&\le C\int_0^{t/2} \int_0^1 (t-\lambda s)^{-\gamma_q-3/2} s^{-1/2} (1+s)^{-1/2} \log (2+s) d\lambda ds\\
	&+ C t^{-\gamma_q-3/2} \int_0^t s^{-1/2} (1+s)^{-1/2} \log (2+s)ds\\
	&+C\int_{t/2}^t ((t-s)^{-1/2} + t^{-1/2}) s^{-\gamma_q-2} \log s ds
	 = O(t^{-\gamma_q-3/2} (\log t)^2)
\end{aligned}
\]
as $t \to +\infty$.
Here, the logarithms in Corollary \ref{cor-asymplow-wt} generated the extra logarithm $(\log t)^2$.
By the way,
\[
	\| (\theta - \Theta_0 - \Theta_1) (t) \|_{L^q (\mathbb{R}^2)} \le \| (\theta - \Theta_0 - \Theta_1 -\Theta_2) (t) \|_{L^q (\mathbb{R}^2)} + \| \Theta_2 (t) \|_{L^q (\mathbb{R}^2)} = O (t^{-\gamma_q-1})
\]
as $t \to +\infty$ and, thus, the logarithmic evolution is removed from the sharp estimate in Proposition \ref{prop-qg}.
Therefore,
\[
	\| (\theta - \Theta_0 - \Theta_1) (t) \|_{L^q (\mathbb{R}^2)}
	\le
	C t^{-\gamma_q-1/2} (1+t)^{-1/2}
\]
holds and, then, the logarithms on Corollary \ref{cor-asymplow-wt} also are removed:
\begin{equation}\label{corr-cor-asymplow}
	\biggl\| |x|^\mu \biggl( \theta\mathcal{R}^\bot\theta - \sum_{m_1+m_2=0}^1 \Theta_{m_1} \mathcal{R}^\bot \Theta_{m_2} \biggr) (t) \biggr\|_{L^q (\mathbb{R}^2)}
	=
	O (t^{-\gamma_q-2+\mu/2})
\end{equation}
as $t \to +\infty$ for $1 \le q < \infty$ and $0 \le \mu \le 2$.
Applying this instead of Corollary \ref{cor-asymplow-wt} in the above procedure, we see that $\| r_3 (t) \|_{L^q (\mathbb{R}^2)} = O (t^{-\gamma_q-3/2} \log t)$ as $t \to +\infty$ and confirm the assertion with $\mu = 0$.
Next, we show one with $\mu = 1$. 
Since $\| |x| (\theta-\Theta_0-\Theta_1-\Theta_2) \|_{L^q (\mathbb{R}^2)} \le \| |x| r_3^0 \|_{L^q (\mathbb{R}^2)} + \| |x| r_3 \|_{L^q (\mathbb{R}^2)}$ and the linear part fulfills that $\| |x| r_3^0 (t) \|_{L^q (\mathbb{R}^2)} = O (t^{-\gamma_q-1})$ as $t \to +\infty$, we estimate $\| |x|r_3 \|_{L^q (\mathbb{R}^2)}$.
The first term of $r_3$ is handled easily by using \eqref{corr-cor-asymplow}.
We separate the domain of second term as $(0,t)\times\mathbb{R}^2 = Q_1 \cup Q_2 = Q_3 \cup Q_4$ for
\[
\begin{aligned}
	&Q_1 = (0,t/2] \times \mathbb{R}^2\quad\text{and}\quad
	Q_2 = (t/2,t) \times \mathbb{R}^2,
\end{aligned}
\]
and
\[
\begin{aligned}
	&Q_3 = (0,t) \times \{ y \in \mathbb{R}^2 \mid |y| \le |x|/2 \}\quad\text{and}\quad
	Q_4 = (0,t) \times \{ y \in \mathbb{R}^2 \mid |y| > |x|/2 \}.
\end{aligned}
\]
Then the second term of $r_3$ also is separated as
\[
\begin{aligned}
	&\int_0^t \int_{\mathbb{R}^2} \biggl( \nabla G(t-s,x-y) - \sum_{|\beta|=0}^1 \nabla^\beta \nabla G(t,x) (-y)^\beta \biggr)
	\cdot \biggl( \theta\mathcal{R}^\bot\theta - \sum_{m_1+m_2=0}^1 \Theta_{m_1}\mathcal{R}^\bot\Theta_{m_2} \biggr) (s,y) dyds\\
	&=
	r_3^1 + r_3^2 + r_3^3 + r_3^4
\end{aligned}
\]
for
\[
\begin{aligned}
	&r_3^h (t)
	=
	-\iint_{Q_h} 
		(\nabla G(t-s,x-y) - \nabla G(t,x-y))
		\cdot \biggl( \theta\mathcal{R}^\bot\theta - \sum_{m_1+m_2=0}^1 \Theta_{m_1}\mathcal{R}^\bot\Theta_{m_2} \biggr) (s,y)
	dyds
\end{aligned}
\]
for $h = 1$ and $2$, and
\[
\begin{aligned}
	&r_3^h (t)
	=
	-\iint_{Q_h}
		\biggl( \nabla G(t,x-y) - \sum_{|\beta|=0}^1 \nabla^\beta \nabla G(t,x) (-y)^\beta \biggr)\\
		&\hspace{15mm}\cdot \biggl( \theta\mathcal{R}^\bot\theta - \sum_{m_1+m_2=0}^1 \Theta_{m_1}\mathcal{R}^\bot\Theta_{m_2} \biggr) (s,y)
	dyds
\end{aligned}
\]
for $h = 3$ and $4$.
Taylor theorem with \eqref{corr-cor-asymplow} yields for $r_3^1$ and $r_3^3$ that
\[
\begin{aligned}
	&\| |x| r_3^1 (t) \|_{L^q (\mathbb{R}^2)}\\
	&\le
	C\int_0^{t/2} \int_0^1 \| |x| \partial_t \nabla G(t-\lambda s) \|_{L^q (\mathbb{R}^2)} s \biggl\| \biggl( \theta\mathcal{R}^\bot\theta - \sum_{m_1+m_2=0}^1 \Theta_{m_1}\mathcal{R}^\bot\Theta_{m_2} \biggr) (s) \biggr\|_{L^1 (\mathbb{R}^2)} d\lambda ds\\
	&+
	C\int_0^{t/2} \int_0^1 \| \partial_t \nabla G(t-\lambda s) \|_{L^q (\mathbb{R}^2)} s \biggl\| |y| \biggl( \theta\mathcal{R}^\bot\theta - \sum_{m_1+m_2=0}^1 \Theta_{m_1}\mathcal{R}^\bot\Theta_{m_2} \biggr) (s) \biggr\|_{L^1 (\mathbb{R}^2)} d\lambda ds\\
	&\le
	C \int_0^{t/2} \int_0^1 (t-\lambda s)^{-\gamma_q-1} s^{-1/2} \left( (1+s)^{-1/2} + (t-\lambda s)^{-1/2} \right) d\lambda ds
	\le
	C t^{-\gamma_q-1} \log (2+t)
\end{aligned}
\]
and
\[
\begin{aligned}
	&\| |x| r_3^3 (t) \|_{L^q (\mathbb{R}^2)}
	\le
	C \sum_{|\beta|=2} \| |x| \nabla^\beta \nabla G(t) \|_{L^q (\mathbb{R}^2)} \int_0^t \biggl\| y^\beta \biggl( \theta\mathcal{R}^\bot\theta - \sum_{m_1+m_2=0}^1 \Theta_{m_1}\mathcal{R}^\bot\Theta_{m_2} \biggr) (s) \biggr\|_{L^1 (\mathbb{R}^2)} d\lambda ds\\
	&\le
	C t^{-\gamma_q-1} \int_0^t s^{-1/2} (1+s)^{-1/2} ds
	\le
	C t^{-\gamma_q-1} \log (2+t).
\end{aligned}
\]
Taylor theorem performs partially for $r_3^4$ as
\[
\begin{aligned}
	&\| |x| r_3^4 (t) \|_{L^q (\mathbb{R}^2)}
	\le
	C \sum_{|\beta|=1} \| \nabla^\beta \nabla G(t) \|_{L^q (\mathbb{R}^2)} \int_0^t \biggl\| |y| y^\beta \biggl( \theta\mathcal{R}^\bot\theta - \sum_{m_1+m_2=0}^1 \Theta_{m_1}\mathcal{R}^\bot\Theta_{m_2} \biggr) (s) \biggr\|_{L^1 (\mathbb{R}^2)}
	ds\\
	&\le
	C t^{-\gamma_q-1} \log (2+t)
\end{aligned}
\]
for $1 \le q \le \infty$.
The remainder term $r_3^2$ is estimated directly as
\[
\begin{aligned}
	&\| |x| r_3^2 (t) \|_{L^q (\mathbb{R}^2)}
	\le
	C \int_{t/2}^t (\| |x| \nabla G(t-s) \|_{L^1 (\mathbb{R}^2)} + \| |x| \nabla G(t) \|_{L^1 (\mathbb{R}^2)})\\
	 &\hspace{15mm}
	\biggl\| \biggl( \theta\mathcal{R}^\bot\theta - \sum_{m_1+m_2=0}^1 \Theta_{m_1}\mathcal{R}^\bot\Theta_{m_2} \biggr) (s) \biggr\|_{L^q (\mathbb{R}^2)}
	ds\\
	&+C \int_{t/2}^t (\| \nabla G(t-s) \|_{L^1 (\mathbb{R}^2)} + \| \nabla G(t) \|_{L^1 (\mathbb{R}^2)})
	\biggl\| |x| \biggl( \theta\mathcal{R}^\bot\theta - \sum_{m_1+m_2=0}^1 \Theta_{m_1}\mathcal{R}^\bot\Theta_{m_2} \biggr) (s) \biggr\|_{L^q (\mathbb{R}^2)}
	ds\\
	&\le
	C \int_{t/2}^t s^{-\gamma_q-3/2} (1+s)^{-1/2} ds
	 + C \int_{t/2}^t ((t-s)^{-1/2} + t^{-1/2}) s^{-\gamma_q-1} (1+s)^{-1/2} ds
	\le
	C t^{-\gamma_q-1}
\end{aligned}
\]
for $1 \le q < \infty$ and
\[
\begin{aligned}
	&\| |x| r_3^2 (t) \|_{L^\infty (\mathbb{R}^2)}\\
	&\le
	C \int_{t/2}^t (\| |x| \nabla G(t-s) \|_{L^p (\mathbb{R}^2)} + \| |x| \nabla G(t) \|_{L^p (\mathbb{R}^2)})
	\biggl\| \biggl( \theta\mathcal{R}^\bot\theta - \sum_{m_1+m_2=0}^1 \Theta_{m_1}\mathcal{R}^\bot\Theta_{m_2} \biggr) (s) \biggr\|_{L^{p'} (\mathbb{R}^2)}
	ds\\
	&+C \int_{t/2}^t (\| \nabla G(t-s) \|_{L^p (\mathbb{R}^2)} + \| \nabla G(t) \|_{L^p (\mathbb{R}^2)})
	\biggl\| |x| \biggl( \theta\mathcal{R}^\bot\theta - \sum_{m_1+m_2=0}^1 \Theta_{m_1}\mathcal{R}^\bot\Theta_{m_2} \biggr) (s) \biggr\|_{L^{p'} (\mathbb{R}^2)}
	ds\\
	&\le
	C \int_{t/2}^t \left( (t-s)^{-\gamma_p} + t^{-\gamma_p} \right) s^{-5/2+\gamma_p} (1+s)^{-1/2} ds\\
	& + C \int_{t/2}^t ((t-s)^{-\gamma_p-1/2} + t^{-\gamma_p-1/2}) s^{-2+\gamma_p} (1+s)^{-1/2} ds
	\le
	C t^{-2}
\end{aligned}
\]
for supplementary $1 < p < 2$.
Thus, we see that $\| |x| r_3 (t) \|_{L^q (\mathbb{R}^2)} = O (t^{-\gamma_q-1} \log t)$ as $t \to + \infty$ and, then, we confirm the assertion with $\mu = 1$ and complete the proof.
By the way, the singularity as $t \to +0$ comes from $\| |x|^\mu \Theta_2 (t) \|_{L^q (\mathbb{R}^2)} = t^{-\gamma_q-1+\mu/2} \| |x|^\mu \Theta_2 (1) \|_{L^q (\mathbb{R}^2)}$.
\end{proof}
In the derivation of Corollary \ref{cor-asymplow-wt}, if we use Proposition \ref{prop-asympmid-wt} instead of Proposition \ref{prop-qg}, we obtain the following estimate.
%
\begin{corollary}\label{cor-asympmid}
Let $\theta_0 \in L^1 (\mathbb{R}^2) \cap L^\infty (\mathbb{R}^2)$ and $|x|^3 \theta_0 \in L^1 (\mathbb{R}^2)$.
Then
\[
	\biggl\| |x|^\mu \biggl( \theta \mathcal{R}^\bot \theta - \sum_{m_1+m_2=0}^2 \Theta_{m_1} \mathcal{R}^\bot \Theta_{m_2} \biggr) (t) \biggr\|_{L^q (\mathbb{R}^2)}
	= O (t^{-\gamma_q-\frac52+\frac\mu2} \log t)
\]
as $t \to +\infty$ holds for $1 \le q < \infty$ and $0 \le \mu \le 3$.
\end{corollary}

However, finally, the logarithms are removed completely.
We also remark the singularities of $|x|^\mu \theta$ and $|x|^\mu \Theta_2$ as $t \to +0$ when we apply this corollary.
	
\section{Proof of main theorem}\label{sect3}
We further expand $r_3$ introduced in \eqref{r3}.
Then,
\[
\begin{aligned}
	&r_3 (t) = - \sum_{m_1+m_2=2} \sum_{|\beta|=1} \nabla^\beta \nabla G(t) \cdot \int_t^\infty \int_{\mathbb{R}^2} y^\beta (\Theta_{m_1} \mathcal{R}^\bot \Theta_{m_2})(s,y) dyds\\
	&- \sum_{2l+|\beta|=2} \frac{\partial_t^l \nabla^\beta \nabla G(t)}{\beta!} \cdot \int_0^t \int_{\mathbb{R}^2}
		(-s)^l y^\beta \biggl( \theta\mathcal{R}^\bot\theta - \sum_{m_1+m_2=0}^1 \Theta_{m_1} \mathcal{R}^\bot \Theta_{m_2} \biggr) (s,y)
	dyds\\
	&-\sum_{|\beta| = 1} \nabla^\beta \nabla G(t) \cdot \int_t^\infty \int_{\mathbb{R}^2}
		y^\beta \biggl( \theta \mathcal{R}^\bot \theta - \sum_{m_1+m_2=0}^2 \Theta_{m_1} \mathcal{R}^\bot \Theta_{m_2} \biggr)(s,y)
	dyds\\
	&- \int_0^t \int_{\mathbb{R}^2}
		\biggl( \nabla G(t-s,x-y) - \sum_{2l+|\beta|=0}^2 \frac{\partial_t^l \nabla^\beta \nabla G(t,x)}{\beta!} (-s)^l (-y)^\beta \biggr)\\
		 &\hspace{15mm}
        \cdot
		\biggl( \theta\mathcal{R}^\bot\theta - \sum_{m_1+m_2=0}^1 \Theta_{m_1}\mathcal{R}^\bot\Theta_{m_2} \biggr) (s,y)
	dyds\\
	&= - \sum_{2l+|\beta|=2} \frac{\partial_t^l \nabla^\beta \nabla G(t)}{\beta!} \cdot \int_0^\infty \int_{\mathbb{R}^2}
		(-s)^l y^\beta \biggl( \theta\mathcal{R}^\bot\theta - \sum_{m_1+m_2=0}^1 \Theta_{m_1} \mathcal{R}^\bot \Theta_{m_2} \biggr) (s,y)
	dyds
	 + J_3^{\mathrm{L}} + J_3^{\mathrm{H}} + r_4
\end{aligned}
\]
for
\begin{equation}\label{J3L}
\begin{aligned}
	&J_3^{\mathrm{L}} (t) = -\sum_{m_1+m_2=2} \sum_{|\beta|=1} \nabla^\beta \nabla G(t) \cdot \int_t^\infty \int_{\mathbb{R}^2} y^\beta (\Theta_{m_1} \mathcal{R}^\bot \Theta_{m_2})(s,y) dyds
\end{aligned}
\end{equation}
and
\begin{equation}\label{J3H}
\begin{aligned}
	&J_3^{\mathrm{H}}
	=
	- \sum_{m_1+m_2 = 2} \int_0^t \int_{\mathbb{R}^2}
		\biggl( \nabla G(t-s,x-y) - \sum_{2l+|\beta|=0}^2 \frac{\partial_t^l \nabla^\beta \nabla G(t,x)}{\beta!} (-s)^l (-y)^\beta \biggr)\\
		 &\hspace{20mm}
		\cdot (\Theta_{m_1}\mathcal{R}^\bot\Theta_{m_2}) (s,y)
	dyds
\end{aligned}
\end{equation}
of the original forms, and
\[
\begin{aligned}
	&r_4 (t) = 
	-\sum_{2l+|\beta|=2} \frac{\partial_t^l \nabla^\beta \nabla G(t)}{\beta!} \cdot \int_t^\infty \int_{\mathbb{R}^2}
		(-s)^l y^\beta \biggl( \theta\mathcal{R}^\bot\theta - \sum_{m_1+m_2=0}^1 \Theta_{m_1} \mathcal{R}^\bot \Theta_{m_2} \biggr) (s,y)
	dyds\\
	&-\sum_{|\beta| = 1} \nabla^\beta \nabla G(t) \cdot \int_t^\infty \int_{\mathbb{R}^2}
		y^\beta\biggl( \theta \mathcal{R}^\bot \theta - \sum_{m_1+m_2=0}^2 \Theta_{m_1} \mathcal{R}^\bot \Theta_{m_2} \biggr)(s,y)
	dyds\\
	&-\int_0^t \int_{\mathbb{R}^2}
		\biggl( \nabla G(t-s,x-y) - \sum_{2l+|\beta|=0}^2 \frac{\partial_t^l \nabla^\beta \nabla G(t,x)}{\beta!} (-s)^l (-y)^\beta \biggr)\\
		 &\hspace{15mm}
        \cdot
		\biggl( \theta\mathcal{R}^\bot\theta - \sum_{m_1+m_2=0}^2 \Theta_{m_1}\mathcal{R}^\bot\Theta_{m_2} \biggr) (s,y)
	dyds.
\end{aligned}
\]
Here, Taylor theorem together with Lebesgue convergence theorem guarantees that $J_3^{\mathrm{H}} \in C((0,\infty),$ $L^1(\mathbb{R}^2))$.
This is required to use Lemma \ref{lem-discuss} to convert this $J_3 = J_3^{\mathrm{L}} + J_3^{\mathrm{H}}$ to the form \eqref{J3pre} later.
Now we vertified the expansion that
\[
	\theta = \Theta_0 + \Theta_1 + \Theta_2 + \Theta_3 + J_3 + r_4^0 + r_4
\]
for
\[
	r_4^0 (t) = \int_{\mathbb{R}^2} \biggl( G(t,x-y) - \sum_{|\alpha|=0}^3 \frac{\nabla^\alpha G(t,x)}{\alpha!} (-y)^\alpha \biggr) \theta_0 (y) dy
\]
and the above $J_3$ and $r_4$.
We firstly confirm that the coefficient of $\Theta_3$ converges.
Since \eqref{phantom1} says that
\[
	\int_{\mathbb{R}^2} (-y)^\beta (\Theta_{m_1} \mathcal{R}^\bot \Theta_{m_2}) dy = 0
\]
for $2l+|\beta|=2$ and $m_1 + m_2 = 2$, we see from Corollary \ref{cor-asympmid} that
\[
\begin{aligned}
	&\biggl| \int_0^\infty \int_{\mathbb{R}^2}
		(-s)^l (-y)^\beta \biggl( \theta\mathcal{R}^\bot\theta - \sum_{m_1+m_2=0}^1 \Theta_{m_1} \mathcal{R}^\bot \Theta_{m_2} \biggr) (s,y)
	dyds \biggr|\\
	&= \biggl| \int_0^\infty \int_{\mathbb{R}^2}
		(-s)^l (-y)^\beta \biggl( \theta\mathcal{R}^\bot\theta - \sum_{m_1+m_2=0}^2 \Theta_{m_1} \mathcal{R}^\bot \Theta_{m_2} \biggr) (s,y)
	dyds \biggr|\\
	&\le
	C \int_0^\infty s^{-1/2} (1+s)^{-1} \log (2+s) ds < \infty.
\end{aligned}
\]
Here, only the weak singularity $s^{-1/2}$ as $s \to +0$ is observed since $\Theta_{m_1} \mathcal{R}^\bot \Theta_{m_2}$ for $m_1 + m_2 = 2$ are phantoms.
For the same reason, the coefficient of first term of $r_4$ is treated as
\[
\begin{aligned}
	&\int_t^\infty \int_{\mathbb{R}^2}
		(-s)^l y^\beta \biggl( \theta\mathcal{R}^\bot\theta - \sum_{m_1+m_2=0}^1 \Theta_{m_1} \mathcal{R}^\bot \Theta_{m_2} \biggr) (s,y)
	dyds\\
	&=
	\int_t^\infty \int_{\mathbb{R}^2}
		(-s)^l y^\beta \biggl( \theta\mathcal{R}^\bot\theta - \sum_{m_1+m_2=0}^2 \Theta_{m_1} \mathcal{R}^\bot \Theta_{m_2} \biggr) (s,y)
	dyds
	= O(t^{-1/2} \log t)
\end{aligned}
\]
as $t \to +\infty$.
As Corollary \ref{cor-asymplow-wt} treats $r_3$ in the proof of Proposition \ref{prop-asympmid-wt}, Corollary \ref{cor-asympmid} also estimates the other terms of $r_4$ as
\begin{equation}\label{wtr4pre}
	\| |x|^\mu r_4 (t) \|_{L^q (\mathbb{R}^2)} = O (t^{-\gamma_q-2+\mu/2} (\log t)^2)
\end{equation}
and then
\[
	\| |x|^\mu (\theta-\Theta_0-\Theta_1-\Theta_2-\Theta_3-J_3) (t) \|_{L^q (\mathbb{R}^2)}
	=
	O (t^{-\gamma_q-2+\mu/2} (\log t)^2)
\]
as $t \to +\infty$ for $1 \le q \le \infty$ and $0 \le \mu \le 1$.
This removes the logarithms from 
Corollary \ref{cor-asympmid}, i.e.,
\[
	\biggl\| |x|^\mu \biggl( \theta \mathcal{R}^\bot \theta - \sum_{m_1+m_2=0}^2 \Theta_{m_1} \mathcal{R}^\bot \Theta_{m_2} \biggr) (t) \biggr\|_{L^q (\mathbb{R}^2)}
	= O (t^{-\gamma_q-\frac52+\frac\mu2})
\]
as $t \to +\infty$ for $1 \le q < \infty$ and $0 \le \mu \le 3$ holds.
Applying this fact instead of Corollary \ref{cor-asympmid} into the derivation process of \eqref{wtr4pre} mitigates the logarithmic evolution such as
\[
	\| r_4 (t) \|_{L^q (\mathbb{R}^2)} = O (t^{-\gamma_q-2} \log t)
\]
as $t \to +\infty$.
This is the desired sharp decay-rate.

At last, we convert the above original $J_3$ to the form \eqref{J3pre}.
Omitting the phantoms from \eqref{J3H} yields that
\begin{equation}\label{J3bs}
\begin{aligned}
	&J_3^{\mathrm{H}} (t) =
	- \sum_{m_1+m_2 = 2} \int_0^t \int_{\mathbb{R}^2}
		\biggl( \nabla G(t-s,x-y) + \sum_{|\beta|=1} \nabla^\beta \nabla G(t,x) y^\beta \biggr)\\
		 &\hspace{20mm}
		\cdot (\Theta_{m_1}\mathcal{R}^\bot\Theta_{m_2}) (s,y)
	dyds.
\end{aligned}
\end{equation}
We see from the representation
\begin{equation}\label{Str}
	\mathcal{R}^\bot \varphi = \frac{\sqrt{\pi}}{2\pi} \int_0^\infty \int_0^\infty \lambda^{-3/2} e^{-\frac1{4\lambda}} \nabla^\bot e^{\sigma^2\lambda\Delta} \varphi d\lambda d\sigma
\end{equation}
with the integration by parts that
\[
\begin{aligned}
	&\nabla e^{(t-s)\Delta}\cdot (\Theta_0\mathcal{R}^\bot\Theta_2) (s)
	=
	- \frac{\sqrt{\pi}}{2\pi} \int_0^\infty \int_0^\infty \lambda^{-3/2} e^{-\frac1{4\lambda}} \nabla e^{(t-s)\Delta}\cdot \left(\Theta_2 (s+\sigma^2\lambda) \nabla^\bot \Theta_0 (s)\right) d\lambda d\sigma
\end{aligned}
\]
since $\nabla^\bot \cdot \nabla = 0$.
For the idea of representation \eqref{Str}, see \cite{Str}.
Hence, the first part of \eqref{J3bs} is converted as
\begin{equation}\label{J3-1}
\begin{aligned}
	&\sum_{m_1+m_2=2} \nabla e^{(t-s)\Delta}\cdot  (\Theta_{m_1} \mathcal{R}^\bot \Theta_{m_2}) (s)\\
	&=
	\frac{\sqrt{\pi}}{2\pi} \int_0^\infty \int_0^\infty \lambda^{-3/2} e^{-\frac1{4\lambda}} \nabla e^{(t-s)\Delta}\cdot \left( \Theta_2 (s) \nabla^\bot \Theta_0 (s+\sigma^2\lambda) - \Theta_2 (s+\sigma^2\lambda) \nabla^\bot \Theta_0 (s) \right) d\lambda d\sigma\\
	&+
	\frac{\sqrt{\pi}}{2\pi}  \int_0^\infty \int_0^\infty \lambda^{-3/2} e^{-\frac1{4\lambda}} \nabla e^{(t-s)\Delta}\cdot  (\Theta_1 (s) \nabla^\bot \Theta_1 (s+\sigma^2\lambda)) d\lambda d\sigma\\
	&=
	\frac{\sqrt{\pi}M_0}{2\pi} \sum_{i,j=1}^2 M_2^{ij} \int_0^\infty \int_0^\infty \lambda^{-3/2} e^{-\frac1{4\lambda}} \nabla e^{(t-s)\Delta}\cdot \bigl( \partial_i \partial_j G (s) \nabla^\bot G (s+\sigma^2\lambda)\\
	 &\hspace{15mm} 
    - \partial_i \partial_j G (s+\sigma^2\lambda) \nabla^\bot G (s) \bigr) d\lambda d\sigma\\
	&+
	\frac{\sqrt{\pi}}{2\pi} \sum_{i,j=1}^2 M_1^i M_1^j \int_0^\infty \int_0^\infty \lambda^{-3/2} e^{-\frac1{4\lambda}} \nabla e^{(t-s)\Delta}\cdot  (\partial_i G (s) \nabla^\bot \partial_j G (s+\sigma^2\lambda)) d\lambda d\sigma.
\end{aligned}
\end{equation}
We calculate these summations by dividing each into partial sums.
Since
\[
\begin{aligned}
	&\sum_{i=1}^2 M_2^{ii} \left( \partial_i^2 G(s) \nabla^\bot G(s+\sigma^2\lambda) - \partial_i^2 G(s+\sigma^2\lambda) \nabla^\bot G(s) \right)\\
	&=
	- \frac{\sigma^2\lambda}{32\pi s^2 (s+\sigma^2\lambda)^2 (2s+\sigma^2\lambda)} \sum_{i=1}^2 M_2^{ii} y_i^2 y^\bot G \left( s - \frac{s^2}{2s+\sigma^2\lambda} \right),
\end{aligned}
\]
we see that
\[
\begin{aligned}
	&\sum_{i=1}^2 M_2^{ii} \nabla e^{(t-s)\Delta}\cdot  \left( \partial_i^2 G(s) \nabla^\bot G(s+\sigma^2\lambda) - \partial_i^2 G(s+\sigma^2\lambda) \nabla^\bot G(s) \right)\\
	&=
	- \frac{\sigma^2\lambda}{32\pi s^2 (s+\sigma^2\lambda)^2 (2s+\sigma^2\lambda)} \sum_{i=1}^2 M_2^{ii} \int_{\mathbb{R}^2}
		\nabla G(t-s,x-y) \cdot y_i^2 y^\bot G\left( s - \frac{s^2}{2s+\sigma^2\lambda},y \right)
	dy\\
	&=
	- \frac{\sigma^2\lambda}{64\pi^2 s^2 (s+\sigma^2\lambda)^2 (2s+\sigma^2\lambda)} \sum_{i=1}^2 M_2^{ii}
	\mathcal{F}^{-1} \left[ e^{-(t-s)|\xi|^2} (\xi\cdot \nabla^\bot) \partial_i^2 e^{-( s - \frac{s^2}{2s+\sigma^2\lambda})|\xi|^2} \right]\\
	&=
	\frac{(M_2^{22} - M_2^{11})\sigma^2\lambda}{8\pi^2 (2s+\sigma^2\lambda)^3}
	\mathcal{F}^{-1} \left[ \xi_1 \xi_2 e^{-(t-\frac{s^2}{2s+\sigma^2\lambda}) |\xi|^2} \right]
	=
	\frac{(M_2^{11} - M_2^{22})\sigma^2\lambda}{4\pi (2s+\sigma^2\lambda)^3} \partial_1 \partial_2 G \left( t-\frac{s^2}{2s+\sigma^2\lambda} \right)\\
	&= \frac{(M_2^{11}-M_2^{22})}{4\pi} \sum_{k=0}^\infty \frac{\sigma^2\lambda s^{2k}}{(2s+\sigma^2\lambda)^{k+3}}
	\frac{\partial_1 \partial_2 (-\Delta)^k G (t)}{k!}.
\end{aligned}
\]
Here, we expanded $\partial_1\partial_2 G(t-\frac{s^2}{2s+\sigma^2\lambda})$ around $-\frac{s^2}{2s+\sigma^2\lambda} = 0$.
Remark that $t-\frac{s^2}{2s+\sigma^2\lambda} > \frac{t}2 \gg 0$ for $(s,\lambda,\sigma) \in (0,t) \times (0,\infty)^2$.
Similarly,
\[
\begin{aligned}
	&\sum_{i\neq j} M_2^{ij} \nabla e^{(t-s)\Delta}\cdot  \left(\partial_i \partial_j G(s) \nabla^\bot G(s+\sigma^2\lambda) - \partial_i \partial_j G(s+\sigma^2\lambda) \nabla^\bot G(s) \right)\\
	&=
	(M_2^{12} + M_2^{21}) \nabla e^{(t-s)\Delta}\cdot  \left( \partial_1 \partial_2 G(s) \nabla^\bot G(s+\sigma^2\lambda) - \partial_1 \partial_2 G(s+\sigma^2\lambda) \nabla^\bot G(s) \right)\\
	&=
	- \frac{(M_2^{12}+M_2^{21}) \sigma^2\lambda}{32\pi s^2 (s+\sigma^2\lambda)^2 (2s+\sigma^2\lambda)} \int_{\mathbb{R}^2}
		y^\bot\cdot\nabla G(t-s,x-y) y_1 y_2 G \left( s - \frac{s^2}{2s+\sigma^2\lambda}, y \right)
	dy
\end{aligned}
\]
\[
\begin{aligned}
	&=
	- \frac{(M_2^{12} + M_2^{21}) \sigma^2\lambda}{64\pi^2 s^2 (s+\sigma^2\lambda)^2 (2s+\sigma^2\lambda)} \mathcal{F}^{-1} \left[
		e^{-(t-s)|\xi|^2} (\xi\cdot\nabla^\bot) \partial_1 \partial_2 e^{-(s-\frac{s^2}{2s+\sigma^2\lambda})|\xi|^2} \right]\\
	&=
	\frac{(M_2^{12} + M_2^{21}) \sigma^2\lambda}{16\pi^2 (2s+\sigma^2\lambda)^3} \mathcal{F}^{-1} \left[ (\xi_1^2-\xi_2^2) e^{-(t-\frac{s^2}{2s+\sigma^2\lambda}) |\xi|^2} \right]
	=
	\frac{(M_2^{12} + M_2^{21}) \sigma^2\lambda}{8\pi (2s+\sigma^2\lambda)^3} (\partial_2^2 - \partial_1^2) G \left( t-\frac{s^2}{2s+\sigma^2\lambda} \right)\\
	&=
	\frac{M_2^{12}+M_2^{21}}{8\pi} \sum_{k=0}^\infty \frac{\sigma^2 \lambda s^{2k}}{(2s+\sigma^2\lambda)^{k+3}}
	\frac{(\partial_2^2 - \partial_1^2) (-\Delta)^k G(t)}{k!}.
\end{aligned}
\]
Hence, by using elementary,
we get
\[
\begin{aligned}
	&\sum_{i,j=1}^2 M_2^{ij} \int_0^\infty \int_0^\infty \lambda^{-3/2} e^{-\frac1{4\lambda}} \nabla e^{(t-s)\Delta}\cdot \bigl( \partial_i \partial_j G (s) \nabla^\bot G (s+\sigma^2\lambda)
    - \partial_i \partial_j G (s+\sigma^2\lambda) \nabla^\bot G (s) \bigr) d\lambda d\sigma\\
	&= \frac{M_2^{11}-M_2^{22}}{16} \sum_{k=0}^\infty \frac{(2k+1)!!}{(2k+4)!!} \left( \frac{s}2 \right)^{k-3/2} \frac{\partial_1\partial_2 (-\Delta)^k G(t)}{k!}\\
	&+ \frac{M_2^{12} + M_2^{21}}{32} \sum_{k=0}^\infty \frac{(2k+1)!!}{(2k+4)!!} \left( \frac{s}2 \right)^{k-3/2} \frac{(\partial_2^2-\partial_1^2) (-\Delta)^k G(t)}{k!}
\end{aligned}
\]
for the first part of \eqref{J3-1}.
We treat another summation.
For the first term, we see that
\[
\begin{aligned}
	&(M_1^1)^2 \nabla e^{(t-s)\Delta}\cdot  (\partial_1 G(s) \nabla^\bot \partial_1 G(s+\sigma^2\lambda))
	=
	- (M_1^1)^2 \nabla e^{(t-s)\Delta}\cdot  (\nabla^\bot \partial_1 G(s) \partial_1 G(s+\sigma^2\lambda))\\
	&=
	(M_1^1)^2 \left( \partial_1 e^{(t-s)\Delta} (\partial_2\partial_1 G(s) \partial_1 G(s+\sigma^2\lambda)) - \partial_2 e^{(t-s)\Delta} (\partial_1^2 G(s) \partial_1 G(s+\sigma^2\lambda)) \right)\\
	&=
	- \frac{(M_1^1)^2}{32\pi s^2 (s+\sigma^2\lambda)(2s+\sigma^2\lambda)} \int_{\mathbb{R}^2}
		\partial_1 G(t-s,x-y) y_1^2 y_2 G \left( s - \frac{s^2}{2s+\sigma^2\lambda}, y \right)
	dy\\
	&- \frac{(M_1^1)^2}{16\pi s (s+\sigma^2\lambda)(2s+\sigma^2\lambda)} \int_{\mathbb{R}^2}
		\partial_2 G(t-s,x-y) \left( 1 - \frac{y_1^2}{2s} \right) y_1 G\left( s - \frac{s^2}{2s+\sigma^2\lambda}, y \right)
	dy.
\end{aligned}
\]
Here,
\[
\begin{aligned}
	&\int_{\mathbb{R}^2}
		\partial_1 G(t-s,x-y) y_1^2 y_2 G \left( s - \frac{s^2}{2s+\sigma^2\lambda}, y \right)
	dy\\
	&=
	\frac{2 s^2(s+\sigma^2\lambda)^2}{\pi (2s+\sigma^2\lambda)^2} \mathcal{F}^{-1} \left[
		\left( 1 - \frac{2s(s+\sigma^2\lambda)}{2s+\sigma^2\lambda} \xi_1^2 \right) \xi_1 \xi_2 e^{-(t-\frac{s^2}{2s+\sigma^2 \lambda})|\xi|^2} \right]
\end{aligned}
\]
and
\[
\begin{aligned}
	&\int_{\mathbb{R}^2}
		\partial_2 G(t-s,x-y) \left( 1 - \frac{y_1^2}{2s} \right) y_1 G\left( s - \frac{s^2}{2s+\sigma^2\lambda}, y \right)
	dy\\
	&= - \frac{s (s+\sigma^2\lambda)}{\pi (2s+\sigma^2\lambda)} \mathcal{F}^{-1} \left[
		\left( -1 + \frac{3(s+\sigma^2\lambda)}{2s+\sigma^2\lambda} - \frac{2s(s+\sigma^2\lambda)^2}{(2s+\sigma^2\lambda)^2} \xi_1^2 \right) \xi_1 \xi_2 e^{-(t-\frac{s^2}{2s+\sigma^2 \lambda})|\xi|^2 } \right],
\end{aligned}
\]
and then
\[
\begin{aligned}
	&(M_1^1)^2 \nabla e^{(t-s)\Delta}\cdot  (\partial_1 G(s) \nabla^\bot \partial_1 G(s+\sigma^2\lambda))
	=
	\frac{(M_1^1)^2 \sigma^2\lambda}{16\pi^2 (2s+\sigma^2\lambda)^3} \mathcal{F}^{-1} \left[ \xi_1 \xi_2 e^{-(t-\frac{s^2}{2s+\sigma^2\lambda})|\xi|^2} \right]\\
	&= -\frac{(M_1^1)^2 \sigma^2\lambda}{8\pi (2s+\sigma^2\lambda)^3} \partial_1 \partial_2 G \left( t-\frac{s^2}{2s+\sigma^2\lambda} \right)
	=
	-\frac{(M_1^1)^2}{8\pi} \sum_{k=0}^\infty \frac{\sigma^2\lambda s^{2k}}{(2s+\sigma^2\lambda)^{k+3}}
	\frac{\partial_1\partial_2 (-\Delta)^k G(t)}{k!}.
\end{aligned}
\]
Same procedure gives
\[
\begin{aligned}
	&(M_1^2)^2 \nabla e^{(t-s)\Delta}\cdot  (\partial_2 G(s) \nabla^\bot \partial_2 G(s+\sigma^2\lambda))\\
	&= (M_1^2)^2 \left( \partial_1 e^{(t-s)\Delta} (\partial_2^2 G(s) \partial_2 G(s+\sigma^2\lambda) - \partial_2 e^{(t-s)\Delta} (\partial_1\partial_2 G(s) \partial_2 G(s+\sigma^2\lambda)) \right)\\
	&= \frac{(M_1^2)^2 \sigma^2\lambda}{8\pi (2s+\sigma^2\lambda)^3} \partial_2 \partial_1 G \left( t-\frac{s^2}{2s+\sigma^2\lambda} \right)
	= \frac{(M_1^2)^2}{8\pi} \sum_{k=0}^\infty \frac{\sigma^2\lambda s^{2k}}{(2s+\sigma^2\lambda)^{k+3}}
	\frac{\partial_2\partial_1 (-\Delta)^k G(t)}{k!}.
\end{aligned}
\]
The last remaining terms are treated in the same way such as
\[
\begin{aligned}
	&M_1^1 M_1^2 \nabla e^{(t-s)\Delta}\cdot  \left( \partial_1 G(s) \nabla^\bot \partial_2 G (s+\sigma^2\lambda)
	+ \partial_2 G(s) \nabla^\bot \partial_1 G (s+\sigma^2\lambda) \right)\\
	&= - \frac{M_1^1 M_1^2}{16\pi s (s+\sigma^2\lambda) (2s+\sigma^2\lambda)} \int_{\mathbb{R}^2}
		\partial_1 G(t-s,x-y) y_1 \left( 1 - \frac{y_2^2}{s+\sigma^2\lambda} \right) G\left( s - \frac{s^2}{2s+\sigma^2\lambda},y \right) dy\\
	&+ \frac{M_1^1 M_1^2}{16\pi s (s+\sigma^2\lambda) (2s+\sigma^2\lambda)} \int_{\mathbb{R}^2}
		\partial_2 G(t-s,x-y) y_2 \left( 1 - \frac{y_1^2}{s+\sigma^2\lambda} \right) G\left( s - \frac{s^2}{2s+\sigma^2\lambda},y \right) dy\\
	&=
	\frac{M_1^1 M_1^2 \sigma^2 \lambda}{16 \pi^2 (2s+\sigma^2\lambda)^3} \mathcal{F}^{-1} \left[  (\xi_2^2-\xi_1^2) e^{-(t-\frac{s^2}{2s+\sigma^2\lambda}) |\xi|^2} \right]
	=
	-\frac{M_1^1 M_1^2 \sigma^2 \lambda}{8 \pi (2s+\sigma^2\lambda)^3} ( \partial_2^2 - \partial_1^2) G \left( t-\frac{s^2}{2s+\sigma^2\lambda} \right)\\
	&=
	-\frac{M_1^1 M_1^2}{8\pi} \sum_{k=0}^\infty \frac{\sigma^2\lambda s^{2k}}{(2s+\sigma^2\lambda)^{k+3}}
	\frac{(\partial_2^2 - \partial_1^2) (-\Delta)^k G(t)}{k!}.
\end{aligned}
\]
Thus, for the latter part of \eqref{J3-1},
\[
\begin{aligned}
	&\sum_{i,j=1}^2 M_1^i M_1^j \int_0^\infty \int_0^\infty \lambda^{-3/2} e^{-\frac1{4\lambda}} \nabla e^{(t-s)\Delta}\cdot  (\partial_i G (s) \nabla^\bot \partial_j G (s+\sigma^2\lambda)) d\lambda d\sigma\\
	&= \frac{(M_1^2)^2 - (M_1^1)^2}{32} \sum_{k=0}^\infty  \frac{(2k+1)!!}{k!(2k+4)!!} \left( \frac{s}2 \right)^{k-3/2} \partial_1\partial_2 (-\Delta)^k G(t)\\
	&- \frac{M_1^1 M_1^2}{32} \sum_{k=0}^\infty \frac{(2k+1)!!}{(2k+4)!!} \left( \frac{s}2 \right)^{k-3/2} \frac{(\partial_2^2-\partial_1^2)(-\Delta)^k G(t)}{k!}.
\end{aligned}
\]
Summing up them into \eqref{J3-1} provides that
\begin{equation}\label{J3bs1}
\begin{aligned}
	&\sum_{m_1+m_2=2} \nabla e^{(t-s)\Delta}\cdot  (\Theta_{m_1} \mathcal{R}^\bot \Theta_{m_2}) (s)\\
	&= \frac{\sqrt{\pi}(2M_0(M_2^{11}-M_2^{22})+(M_1^2)^2 - (M_1^1)^2)}{64\pi} \sum_{k=0}^\infty \frac{(2k+1)!!}{(2k+4)!!} \left( \frac{s}2 \right)^{k-3/2} \frac{\partial_1\partial_2 (-\Delta)^k G(t)}{k!}\\
	&+ \frac{\sqrt{\pi}(M_0(M_2^{12} + M_2^{21}) - M_1^1M_1^2)}{64\pi} \sum_{k=0}^\infty \frac{(2k+1)!!}{(2k+4)!!} \left( \frac{s}2 \right)^{k-3/2} \frac{(\partial_2^2-\partial_1^2) (-\Delta)^k G(t)}{k!}.
\end{aligned}
\end{equation}
Now we employ Lemma \ref{lem-discuss}.
By virtue of the scales, we can choose $t = 1$.
We put $g(s) = s^{-3/2} h + E(s)$ for
\[
\begin{aligned}
	&g(s) = -\sum_{m_1+m_2=2} \nabla e^{(1-s)\Delta} \cdot (\Theta_{m_1} \mathcal{R}^\bot \Theta_{m_2}) (s),\\
	&h =  \frac{\sqrt{2\pi}(2M_0(M_2^{22}-M_2^{11})+(M_1^1)^2 - (M_1^2)^2)}{256\pi} \partial_1\partial_2 G(1)\\
	&+ \frac{\sqrt{2\pi}(M_0(M_2^{12} + M_2^{21}) - M_1^1M_1^2)}{256\pi} (\partial_1^2-\partial_2^2) G(1)
\end{aligned}
\]
and
\[
\begin{aligned}
	&E (s)
	= \frac{\sqrt{\pi}(2M_0(M_2^{22}-M_2^{11})+(M_1^1)^2 - (M_1^2)^2)}{64\pi} \sum_{k=1}^\infty \frac{(2k+1)!!}{(2k+4)!!} \left( \frac{s}2 \right)^{k-3/2} \frac{\partial_1\partial_2 (-\Delta)^k G(1)}{k!}\\
	&+ \frac{\sqrt{\pi}(M_0(M_2^{12} + M_2^{21}) - M_1^1M_1^2)}{64\pi} \sum_{k=1}^\infty \frac{(2k+1)!!}{(2k+4)!!} \left( \frac{s}2 \right)^{k-3/2} \frac{(\partial_1^2-\partial_2^2) (-\Delta)^k G(1)}{k!}.
\end{aligned}
\]
On the other hand the integrand of $J_3^{\mathrm{H}}$ of form \eqref{J3bs} is rearranged as
\[
\begin{aligned}
	&- \sum_{m_1+m_2 = 2} \int_{\mathbb{R}^2}
		\biggl( \nabla G(1-s,x-y) + \sum_{|\beta|=1} \nabla^\beta \nabla G(1,x) y^\beta \biggr)
		\cdot (\Theta_{m_1}\mathcal{R}^\bot\Theta_{m_2}) (s,y)
	dy = g(s) - s^{-3/2} f
\end{aligned}
\]
for
\[
\begin{aligned}
	f= \sum_{m_1+m_2=2} \sum_{|\beta|=1} \nabla^\beta \nabla G(1) \cdot \int_{\mathbb{R}^2}
		y^\beta (\Theta_{m_1} \mathcal{R}^\bot \Theta_{m_2}) (1,y)
	dy,
\end{aligned}
\]
and then, we obtain that
\begin{equation}\label{kakera}
\begin{aligned}
	&\sum_{m_1+m_2=2} \sum_{|\beta| = 1} \nabla^\beta \nabla G(t) \cdot \int_{\mathbb{R}^2} y^\beta (\Theta_{m_1} \mathcal{R}^\bot \Theta_{m_2}) (s,y) dy\\
	&= \frac{\sqrt{2\pi} \left( M_0 (M_2^{12}+M_2^{21}) - M_1^1 M_1^2\right)}{256\pi} s^{-3/2} (\partial_1^2 - \partial_2^2) G(t)\\
	&+ \frac{\sqrt{2\pi}\left( 2M_0 (M_2^{22}-M_2^{11})+(M_1^1)^2-(M_1^2)^2\right)}{256\pi} s^{-3/2} \partial_1 \partial_2 G(t).
\end{aligned}
\end{equation}
Therefore, we get $J_3^{\mathrm{L}}$ of \eqref{J3pre} from \eqref{J3L}.
Of course, \eqref{kakera} could also be confirmed by the $L^2$-theory.
For $J_3^{\mathrm{H}}$ of the form \eqref{J3bs}, we conclude from \eqref{J3bs1} and \eqref{kakera} that
\[
\begin{aligned}
	&J_3^{\mathrm{H}} (t) = \frac{\sqrt{\pi}(2M_0(M_2^{22}-M_2^{11})+(M_1^1)^2 - (M_1^2)^2)}{64\pi} \sum_{k=1}^\infty \frac{(2k+1)!!}{(2k+4)!!} \frac{\partial_1\partial_2 (-\Delta)^k G(t)}{k!}
	\int_0^t \left( \frac{s}2 \right)^{k-3/2} ds\\
	&+ \frac{\sqrt{\pi}(M_0(M_2^{12} + M_2^{21}) + M_1^1M_1^2)}{64\pi} \sum_{k=1}^\infty \frac{(2k+1)!!}{(2k+4)!!} \frac{(\partial_1^2-\partial_2^2) (-\Delta)^k G(t)}{k!}
	\int_0^t \left( \frac{s}2 \right)^{k-3/2} ds
\end{aligned}
\]
and then we see the form \eqref{J3}.
If Taylor's theorem is not applied at each stage, the given form \eqref{J3pre} is obtained.
\begin{remark}\label{rem31}
Our main result is derived from renormalizing $\theta \mathcal{R}^\bot\theta \sim \sum_{m_1 + m_2 =0}^2 \Theta_{m_1} \mathcal{R}^\bot \Theta_{m_2}$.
In this procedure, the pseudo-logarithms
\[
	\int_0^t \int_{\mathbb{R}^2} y^\beta (\Theta_{m_1} \mathcal{R}^\bot \Theta_{m_2}) (1+s,y) dy ds
	= \int_0^t (1+s)^{-1} ds \int_{\mathbb{R}^2} y^\beta (\Theta_{m_1} \mathcal{R}^\bot \Theta_{m_2}) (1,y) dy
\]
for $|\beta| = m_1 + m_2$ vanish since the parity of $\Theta_{m_1} \mathcal{R}^\bot \Theta_{m_2}$ is determined by $m_1 + m_2$ and then $y^\beta (\Theta_{m_1} \mathcal{R}^\bot \Theta_{m_2})$ should be odd-type.
If we further renormalize as $\theta \mathcal{R}^\bot\theta \sim \sum_{m_1 + m_2 =0}^3 \Theta_{m_1} \mathcal{R}^\bot \Theta_{m_2} + J_3 \mathcal{R}^\bot \Theta_0 + \Theta_0 \mathcal{R}^\bot J_3$, then
\[
	\int_0^t \int_{\mathbb{R}^2} y^\beta (J_3 \mathcal{R}^\bot \Theta_0 + \Theta_0 \mathcal{R}^\bot J_3)(1+s,y) dyds
	=
	\int_0^t (1+s)^{-1} ds \int_{\mathbb{R}^2} y^\beta (J_3 \mathcal{R}^\bot \Theta_0 + \Theta_0 \mathcal{R}^\bot J_3)(1,y) dy
\]
for $|\beta| = 3$ contains both parities.
Therefore, we can expect that there is some profile $K_4$ satisfying $\lambda^{2+4} K_4 (\lambda^2 t, \lambda x) = K_4 (t,x)$ for $\lambda > 0$ and
\[
	\theta (t) \sim \Theta_0 (t) + \Theta_1 (t) + \Theta_2 (t) + \Theta_3 (t) + J_3 (t) + K_4 (t) \log t
\]
as $t \to +\infty$.
Precisely,
\[
	K_4 (t) = \sum_{2l+|\beta|=3} \frac{\partial_t^l \nabla^\beta \nabla G(t)}{\beta!} \cdot \int_{\mathbb{R}^2} (-1)^l y^\beta (J_3 \mathcal{R}^\bot \Theta_0 + \Theta_0 \mathcal{R}^\bot J_3)(1,y) dy.
\]
If this expectation is true, then the second assertion of Theorem \ref{thm} is optimal.
Of course, the actual calculation might lead to $K_4 = 0$.
\end{remark}
\begin{remark}\label{rem33}
Readers notice that $J_3 = J_3^{\mathrm{L}} + J_3^{\mathrm{H}}$ originally given by \eqref{J3L} and \eqref{J3H} is formally transformed into
\[
\begin{aligned}
	&J_3 (t) = - \sum_{m_1+m_2 = 2} \int_0^t \nabla e^{(t-s)\Delta}\cdot (\Theta_{m_1} \mathcal{R}^\bot \Theta_{m_2}) (s) ds\\
	&+ \sum_{m_1+m_2=2} \sum_{|\beta|=0}^1 \nabla^\beta \nabla G(t) \cdot \int_0^\infty \int_{\mathbb{R}^2} (-y)^\beta (\Theta_{m_1} \mathcal{R}^\bot \Theta_{m_2}) (s,y) dyds.
\end{aligned}
\]
However, they will also notice that convergence of this term-by-term integral is not guaranteed.
\end{remark}

\end{document}